\documentclass[amstex,12pt,reqno]{amsart}
\usepackage{amsmath, amssymb, amsthm}
\usepackage{mathrsfs}
\usepackage{hyperref}

\usepackage{mathtools}
\mathtoolsset{showonlyrefs=true}

\newtheorem{theorem}{Theorem}[section]
\newtheorem{lemma}[theorem]{Lemma}
\newtheorem{proposition}[theorem]{Proposition}
\newtheorem{corollary}[theorem]{Corollary}

\theoremstyle{definition}
\newtheorem{definition}{Definition}
\newtheorem{remark}{Remark}

\title[Integrable Teich\-m\"ul\-ler space]{Analytic Besov functions, pre-Schwarzian derivatives, 
\\and integrable Teich\-m\"ul\-ler spaces}

\author[K. Matsuzaki]{Katsuhiko Matsuzaki}
\address{Department of Mathematics, School of Education, Waseda University \endgraf
Nishi-Waseda 1-6-1, Shinjuku, Tokyo 169-8050, Japan}
\email{matsuzak@waseda.jp}
\author[H. Wei]{Huaying Wei} 
\address{Center for Applied  Mathematics, Tianjin University \endgraf No. 92 Weijin Road,
Tianjin 300072, PR China}
\email{hywei@tju.edu.cn}


\thanks{Research supported by 
Japan Society for the Promotion of Science (KAKENHI 23K25775 and 23K17656) and the
National Natural Science Foundation of China (grant no. 12271218 and 12571083)}

\begin{document}

\maketitle

\begin{abstract}
We study the embedding of integrable Teichm\"uller spaces $T_p$ into analytic Besov spaces
via pre-Schwarzian derivatives. In contrast to the Bers embedding by Schwarzian derivatives,
a significant difference arises between the cases $p>1$ and $p=1$. In this paper we focus on the case $p=1$ and
extend previous results obtained for $p>1$. 
This provides a unified framework for the complex-analytic theory of 
integrable Teichm\"uller spaces $T_p$ for all $p \geq 1$.
\end{abstract}

\section{Introduction}\label{1}

The integrable Teich\-m\"ul\-ler space has been extensively studied as a subspace of the universal Teich\-m\"ul\-ler space that carries the Weil--Petersson metric and parametrizes the family of Weil--Petersson curves. Bishop's recent characterization of Weil--Petersson curves \cite{Bi} is closely related to this theory from the complex-analytic viewpoint.
Wang \cite{Wa} defined a Dirichlet energy arising from the Loewner ODE that generates SLE and showed that the finiteness of this energy forces the evolving arcs to be Weil--Petersson. Moreover, this Dirichlet energy coincides with the universal Liouville action on the integrable Teichm\"ul-ler space, which serves as a K\"ahler potential for the Weil--Petersson metric.

The integrable Teich\-m\"ul\-ler space $T_2$ was introduced by Cui \cite{Cu}, and its Hilbert manifold structure and Weil--Petersson geometry were subsequently developed by Takhtajan and Teo \cite{TT}; foundational complex-analytic aspects were established by Shen \cite{Sh}. 
This space is the quotient of $M_2(\mathbb H^+)$, the space of square-integrable
Beltrami coefficients on $\mathbb H^+$ (where $\mathbb H^\pm$ denote the upper/lower half-planes),
by Teichm\"uller equivalence,
and is embedded homeomorphically into the Hilbert space
$\mathcal A_2(\mathbb H^-)$ by the Bers embedding.
We recall the definitions of these spaces and mappings below in the case $p=2$. 
The Weil--Petersson metric is induced either from the inner product on $\mathcal A_2(\mathbb H^-)$ or equivalently from the pairing
$$
\int_{\mathbb H^+}\mu(z) \overline{\nu(z)}\frac{dx\,dy}{|{\rm Im}\, z|^2}
$$
for harmonic representatives $\mu, \nu \in M_2(\mathbb H^+)$ of tangent vectors of $T_2$ at the origin.

For a parameter $p \geq 1$, the $p$-integrable Teich\-m\"ul\-ler space $T_p$ is modeled on $p$-integrable Beltrami coefficients on the half-plane. Explicitly, we set
$$
M_p(\mathbb H^+)=\Bigl\{\mu \in M(\mathbb H^+) \mid \Vert \mu \Vert_p=\Bigl(\int_{\mathbb H^+}
|\mu(z)|^p\,\frac{dx\,dy}{|{\rm Im}\, z|^2}\Bigr)^{1/p}<\infty \Bigr\},
$$
where $M(\mathbb H^+)=\{\mu \in L_\infty(\mathbb H^+) \mid \Vert \mu \Vert_\infty<1\}$ is the space of Beltrami coefficients, and define $T_p=\{[\mu] \mid \mu \in M_p(\mathbb H^+)\}$ as the set of Teichm\"uller equivalence classes,
which is included in the universal Teichm\"uller space $T=\{[\mu] \mid \mu \in M(\mathbb H^+)\}$ 
(see Section \ref{5}). The original theory concentrated on $p=2$, was extended to $p \geq 2$ by Guo \cite{Gu},
Tang and Shen \cite{TS}, and further to $p>1$ by Wei and Matsuzaki \cite{WM-3}. In addition, it was proved in \cite{WM-1,WM-4} that the Bers embedding
$\alpha:T_p \to \mathcal A_p(\mathbb H^-)$,
defined by $\alpha([\mu])=S_{F^\mu}$ via the Schwarzian derivative of the normalized conformal homeomorphism 
$F^\mu:\mathbb H^-\to\mathbb C$ with quasiconformal extension to the plane
of dilatation $\mu \in M_p(\mathbb H^+)$, is a homeomorphism onto its image for all $p\geq 1$. Hence $T_p$ inherits a natural complex structure modeled on the Banach space 
$$
\mathcal A_p(\mathbb H^-)=\Bigl\{ \Phi \in {\rm Hol}(\mathbb H^-) \mid \Vert \Phi\Vert_{{\mathcal A}_p}=\Bigl(\int_{\mathbb H^-} |({\rm Im}\,z)^2 \Phi(z)|^p \frac{dx\,dy}{|{\rm Im}\,z|^2}\Bigr)^{1/p}<\infty \Bigr\},
$$
where ${\rm Hol}(\mathbb H^-)$ denotes the holomorphic functions on $\mathbb H^-$.

This paper develops a unified embedding theory for integrable Teich\-m\"ul\-ler spaces via the logarithm of derivative $\log (F^\mu)'$ and the pre-Schwarzian derivative $N_{F^\mu}=(\log (F^\mu)')'$, including the endpoint $p=1$. 
For the universal Teichm\"uller space $T$, this model is
intensively studied by Zhuravlev \cite{Z} on the unit disk $\mathbb D$.
The target on the function side of $T_p$ is the analytic Besov space on $\mathbb H^-$:
for $p>1$,
$$
\mathcal B_p(\mathbb H^-)=\Bigl\{\Phi \in {\rm Hol}(\mathbb H^-) \mid 
\Vert\Phi\Vert_{\mathcal B_p}=\Bigl(\int_{\mathbb H^-}|({\rm Im}\, z)\,\Phi'(z)|^p\,\frac{dx\,dy}{|{\rm Im}\, z|^2}\Bigr)^{1/p}<\infty
\Bigr\},
$$
while for $p \geq 1$ we also set
$$
\mathcal B_p^{\#}(\mathbb H^-)=\Bigl\{\Phi \in {\rm Hol}(\mathbb H^-) \mid 
\Vert\Phi\Vert_{\mathcal B_p^{\#}}=\Bigl(\int_{\mathbb H^-}|({\rm Im}\, z)^2\,\Phi''(z)|^p\,\frac{dx\,dy}{|{\rm Im}\, z|^2}\Bigr)^{1/p}<\infty
\Bigr\}.
$$
Then we define
$$
\widehat{\mathcal B}_p(\mathbb H^-)=\mathcal B_p^{\#}(\mathbb H^-) \cap {\rm BMOA}(\mathbb H^-)
$$
with norm $\Vert \Phi \Vert_{\widehat{\mathcal B}_p}=\Vert \Phi \Vert_{{\mathcal B}_p^{\#}}+\Vert \Phi \Vert_{\rm BMOA}$,
where $\mathrm{BMOA}(\mathbb H^-)$ is the Banach space of holomorphic functions $\Phi$ on $\mathbb H^-$ that
are given by the Poisson integral of BMO functions on the real line $\mathbb R$.
BMOA can also be characterized by Carleson measures (see Section \ref{2}). 

Since $\mathcal B_1(\mathbb H^-)$ collapses to constants, the appropriate target for the pre-Schwarzian at $p=1$ is 
$\widehat{\mathcal B}_1(\mathbb H^-)$. Moreover, for $p>1$, the norms $\Vert\Phi\Vert_{\mathcal B_p}$ and
$\Vert \Phi \Vert_{\widehat{\mathcal B}_p}$ are equivalent.
We also recall the Besov spaces defined on 
$\mathbb D$ and prove that the Cayley transformation yields a Banach space isomorphism
between $\widehat{\mathcal B}_p(\mathbb H^-)$ and $\widehat{\mathcal B}_p(\mathbb D)$ (Theorem \ref{isomorphism}).

Section \ref{3} studies the pre-Schwarzian derivative map
$L:M_p(\mathbb H^+) \to \widehat{\mathcal B}_p(\mathbb H^-)$ given by $L(\mu)=\log (F^\mu)'$.
A direct adaptation of the Schwarzian argument shows the holomorphy of $L$ under certain constraints on $p$; to remove these constraints, we exploit the Schwarzian derivative map
$S:M_p(\mathbb H^+) \to \mathcal A_p(\mathbb H^-)$, $S(\mu)=S_{F^\mu}$,
together with sharp norm estimates (Theorem \ref{P-holo}). Using the existence of a local holomorphic right inverse to $S$, we prove that the canonical holomorphic map
$J:L(M_p(\mathbb H^+)) \to S(M_p(\mathbb H^+))$, $J(\Psi)=\Psi''-\tfrac12(\Psi')^2$,
is in fact biholomorphic (Theorem \ref{SL}). Consequently, the three spaces $M_p(\mathbb H^+)$, $\widehat{\mathcal B}_p(\mathbb H^-)$, and $\mathcal A_p(\mathbb H^-)$ are uniformly linked for all $p \geq 1$ in a manner that extends earlier results (Theorem \ref{conformal}).

In Section \ref{4} we revisit these results on $\mathbb D$ and the exterior unit disk $\mathbb D^*$. Although the Cayley transformation 
identifies $\widehat{\mathcal B}_p(\mathbb H^-)$ with $\widehat{\mathcal B}_p(\mathbb D)$ as Banach spaces, 
the canonical map
$J:L(M_p(\mathbb D^*)) \to S(M_p(\mathbb D^*))$
fails to be injective in this model. A modified statement shows that $J$ is a holomorphic split submersion (Theorem \ref{bundle}). We analyze the fiber structure of $L(M_p(\mathbb D^*))$ over $S(M_p(\mathbb D^*))$, proving that $L(M_p(\mathbb D^*))$ is a real-analytic disk bundle (Theorem \ref{bundle2}); for $p>1$, a global real-analytic section identifies it with the product $S(M_p(\mathbb D^*))\times\mathbb D^*$ (Corollary \ref{product}).

Section \ref{5} discusses the complex Banach manifold structure, the topological group structure, and the Weil--Petersson metric on $T_p$ for $p\ge1$. In parallel with the Bers embedding $\alpha:T_p\to \mathcal A_p(\mathbb H^-)$ via $S$, 
we introduce the pre-Bers embedding
$\beta:T_p\to \widehat{\mathcal B}_p(\mathbb H^-)$
via $L$, and prove that $\alpha$ and $\beta$ induce biholomorphically equivalent complex structures (Theorem \ref{Bersemb}). Moreover, since the Weil--Petersson metric can be regarded as an invariant metric obtained by
right translation of 
the norm on $\alpha(T_p)$, an analogous construction on $\beta(T_p)$ yields an alternative Weil--Petersson metric with similar properties (Theorem \ref{complete}).

Finally, Section \ref{6} compares $T_p$ $(p \geq 1)$ with the Teichm\"uller space $T^\gamma$ $(0<\gamma\leq 1)$ of circle diffeomorphisms whose derivatives are H\"older--Zygmund continuous.
These are 
defined by Beltrami coefficients on $\mathbb D^*$ satisfying $|\mu(z)|=O((|z|-1)^\gamma)$ as $|z|\to1$, corresponding to orientation-preserving circle diffeomorphisms $h$ with $h'\in C^{\gamma}$ (for $\gamma=1$, $h'$ is continuous and satisfies the Zygmund condition). While $T^1\subset T_p$ for every $p>1$ and every $h\in T_1$ is known to be a $C^1$-diffeomorphism, there is no inclusion relation between $T^1$ and $T_1$ (Theorem \ref{inclusion}).


\section{Analytic Besov functions}\label{2}

We denote by $\mathbb H$ either the upper or the lower half-plane. When necessary, we write $\mathbb H^+$ for the upper half-plane and $\mathbb H^-$ for the lower half-plane.

As a generalization of analytic Dirichlet functions (the case $p=2$),
we introduce the following classes of holomorphic functions on $\mathbb H$, which we call
{\it analytic Besov functions}; see \cite[Chapter~5]{Zhu}, where these functions are defined on $\mathbb D$.
As mentioned below, the seminorm $\Vert \cdot \Vert_{\mathcal B_p}$ in the following definition is conformally invariant.
However, the treatment of $\Vert \cdot \Vert_{\mathcal B_p^\#}$ is more delicate.

\begin{definition}
For $p>1$, define the seminorm
$$
\Vert \Phi\Vert_{{\mathcal B}_p}=\Bigl(\int_{\mathbb H} |({\rm Im}\,z)\, \Phi'(z)|^p \frac{dx\,dy}{|{\rm Im}\,z|^2}\Bigr)^{1/p}
$$
for holomorphic functions $\Phi$ on $\mathbb H$. The set 
of all such $\Phi$ with $\Vert \Phi\Vert_{{\mathcal B}_p}<\infty$
is denoted by ${\mathcal B}_p(\mathbb H)$. 
Moreover, for $p \geq 1$, define the seminorm
$$
\Vert \Phi \Vert_{{\mathcal B}_p^{\#}}=\Bigl(\int_{\mathbb H} |({\rm Im}\,z)^2\, \Phi''(z)|^p \frac{dx\,dy}{|{\rm Im}\,z|^2}\Bigr)^{1/p}. 
$$
The set of all such $\Phi$ with $\Vert \Phi\Vert_{{\mathcal B}_p^{\#}}<\infty$
is denoted by ${\mathcal B}_p^{\#}(\mathbb H)$. 
\end{definition}

\begin{remark}
If one applies the seminorm $\Vert \cdot \Vert_{\mathcal B_p}$ with $p=1$, 
then only constant functions $\Phi$ satisfy $\Vert \Phi\Vert_{{\mathcal B}_1}<\infty$; see \cite[p.~132]{Zhu}.
\end{remark}

A holomorphic function $\Phi$ on $\mathbb H$ is called a {\it Bloch function} if the seminorm satisfies
$$
\Vert \Phi\Vert_{{\mathcal B}_\infty}=\sup_{z \in \mathbb H} |({\rm Im}\,z)\, \Phi'(z)|<\infty.
$$
The set of all Bloch functions on $\mathbb H$ is denoted by $\mathcal B_\infty(\mathbb H)$.
Moreover, $\Phi$ is called a {\it BMOA function} if 
$|{\rm Im}\,z|\,|\Phi'(z)|^2\,dx\,dy$ is a Carleson measure on $\mathbb H$.
In general, a (possibly infinite) measure $m$ on $\mathbb H$ is said to be a Carleson measure if
$\sup_{I \subset \mathbb R} m(\widehat I)/|I|<\infty$,
where the supremum is taken over all bounded intervals $I\subset\mathbb R$ and $\widehat I \subset \mathbb H$ denotes
the Carleson box (the square in $\mathbb H$) above $I$. Accordingly,
the BMOA seminorm of $\Phi$ is defined by
$$
\Vert \Phi \Vert_{\rm BMOA}=\Big(\sup_{I \subset \mathbb R} \frac{1}{|I|}
\int_{\widehat I} |({\rm Im}\,z)\,\Phi'(z)|^2\frac{dx\,dy}{|{\rm Im}\,z|}\Big)^{1/2}.
$$
This definition of BMOA is equivalent to requiring that $\Phi$ be holomorphic and given by the Poisson integral of
a BMO function on $\mathbb R$. On the unit disk $\mathbb D$, the corresponding equivalence is well known
(see \cite[Theorem~6.5]{Gi}); on the half-plane $\mathbb H$, it also holds (see \cite[p.~262]{Ga}).
The set of all BMOA functions on $\mathbb H$ is denoted by ${\rm BMOA}(\mathbb H)$.

We next compare the above seminorms. For convenience, we include proofs of the standard estimates.

\begin{proposition}\label{norm}
{\rm (i)} For $1<p \leq q \leq \infty$, 
there exists a constant $c_{p,q}>0$ such that
$\Vert \Phi\Vert_{{\mathcal B}_q} \leq c_{p,q}\,\Vert \Phi\Vert_{{\mathcal B}_p}$.
{\rm (ii)} There exists a constant $c>0$ such that $\Vert \Phi\Vert_{{\mathcal B}_\infty} \leq c\,\Vert \Phi \Vert_{\rm BMOA}$.
{\rm (iii)} For $p>1$, there exists a constant $c'_p>0$ such that
$\Vert \Phi \Vert_{\rm BMOA} \leq c'_p\,\Vert \Phi \Vert_{{\mathcal B}_p}$.
\end{proposition}

\begin{proof}
(i) 
For $z \in \mathbb H$,
let $\Delta(z,|{\rm Im}\,z|/2) \subset \mathbb H$ be the disk centered at $z$ with radius $|{\rm Im}\,z|/2$.
By the integral mean inequality for holomorphic functions and the H\"older inequality,
\begin{align}
|({\rm Im}\,z)\Phi'(z)| &\leq \frac{4}{\pi |{\rm Im}\,z|}\int_{\Delta(z,|{\rm Im}\,z|/2)}|\Phi'(w)|\,du\,dv\\
&\leq \frac{4}{\pi |{\rm Im}\,z|}\Bigl(\frac{\pi|{\rm Im}\,z|^2}{4} \Bigr)^{1-1/p}
\Bigl( \int_{\Delta(z,|{\rm Im}\,z|/2)}|\Phi'(w)|^p\,du\,dv \Bigr)^{1/p}\\
&=\Bigl(\frac{4}{\pi}\Bigr)^{1/p}|{\rm Im}\,z|^{1-2/p}\Bigl(\int_{\Delta(z,|{\rm Im}\,z|/2)}|\Phi'(w)|^p\,du\,dv \Bigr)^{1/p}.\label{mean}
\end{align}
The last line is bounded by
$$
\Bigl(\frac{4}{\pi}\Bigr)^{1/p}|{\rm Im}\,z|^{1-2/p}\Bigl(\frac{2^{p-2}}{|{\rm Im}\,z|^{p-2}} \int_{\Delta(z,|{\rm Im}\,z|/2)}|({\rm Im}\, w)\Phi'(w)|^p \frac{du\,dv}{|{\rm Im}\,w|^2} \Bigr)^{1/p} \leq \frac{2}{\sqrt[p]{\pi}}\Vert \Phi \Vert_{\mathcal B_p},
$$
whence $\Vert \Phi\Vert_{{\mathcal B}_\infty} \leq c_{p,\infty}\,\Vert \Phi\Vert_{{\mathcal B}_p}$
with $c_p=c_{p,\infty}=2/\sqrt[p]{\pi}$.

For $p \leq q<\infty$, we have
$$
\int_{\mathbb H}\Bigl( \frac{|({\rm Im}\,z)\Phi'(z)|}{\Vert \Phi \Vert_{\mathcal B_\infty}}\Bigr)^q \frac{dx\,dy}{|{\rm Im}\,z|^2}
\leq \int_{\mathbb H}\Bigl( \frac{|({\rm Im}\,z)\Phi'(z)|}{\Vert \Phi \Vert_{\mathcal B_\infty}}\Bigr)^p \frac{dx\,dy}{|{\rm Im}\,z|^2},
$$
i.e. $\Vert \Phi \Vert_{\mathcal B_q}^q/\Vert \Phi \Vert_{\mathcal B_\infty}^q \leq \Vert \Phi \Vert_{\mathcal B_p}^p/\Vert \Phi \Vert_{\mathcal B_\infty}^p$.
It follows that
$$
\frac{\Vert \Phi \Vert_{\mathcal B_q}}{\Vert \Phi \Vert_{\mathcal B_\infty}} 
\leq \Bigl(\frac{\Vert \Phi \Vert_{\mathcal B_p}}{\Vert \Phi \Vert_{\mathcal B_\infty}} \Bigr)^{p/q} 
= \Bigl( \frac{1}{c_p}\Bigr)^{p/q}\Bigl(\frac{c_p \Vert \Phi \Vert_{\mathcal B_p}}{\Vert \Phi \Vert_{\mathcal B_\infty}} \Bigr)^{p/q}
\leq \Bigl( \frac{1}{c_p}\Bigr)^{p/q} \frac{c_p \Vert \Phi \Vert_{\mathcal B_p}}{\Vert \Phi \Vert_{\mathcal B_\infty}}. 
$$
Hence
$
\Vert \Phi \Vert_{\mathcal B_q} 
\leq c_{p,q}\,\Vert \Phi \Vert_{\mathcal B_p}$
with $c_{p,q}=c_p^{\,1-p/q}$.

(ii) This is sketched in \cite[p.~92]{Sa}; see also \cite[Corollary~5.2]{Gi}. From \eqref{mean} with $p=2$,
\begin{align}
|({\rm Im}\,z)\Phi'(z)| 
&\leq \frac{2}{\sqrt{\pi}}\Bigl(\int_{\Delta(z,|{\rm Im}\,z|/2)}|\Phi'(w)|^2\,du\,dv \Bigr)^{1/2}\\
&\leq \frac{4}{\sqrt{\pi}}\Bigl(\frac{1}{2|{\rm Im}\,z|}\int_{I^2(z,|{\rm Im}\,z|)}|{\rm Im}\,w|\,|\Phi'(w)|^2\,du\,dv \Bigr)^{1/2},
\end{align}
where $I^2(z,|{\rm Im}\,z|)$ denotes the Carleson box square centered at $z$
above the interval of length $2|{\rm Im}\,z|$ on $\mathbb R$.
Taking the supremum over $z\in\mathbb H$ yields $\Vert \Phi\Vert_{{\mathcal B}_\infty} \leq c\,\Vert \Phi \Vert_{\rm BMOA}$ with $c=4/\sqrt{\pi}$.

(iii) Suppose that $p>2$. For any bounded interval $I\subset\mathbb R$,
the H\"older inequality gives
\begin{align}
&\quad\ \frac{1}{|I|}\int_{\widehat I}|{\rm Im}\,w||\Phi'(w)|^2dudv\\ 
&\leq
\frac{1}{|I|}\Bigl(\int_{\widehat I}|({\rm Im}\,w)\Phi'(w)|^p\frac{dudv}{|{\rm Im}\,w|^2} \Bigr)^{2/p} \cdot
\Bigl(\int_{\widehat I} |{\rm Im}\,w|^{p/(p-2)}\frac{dudv}{|{\rm Im}\,w|^2}\Bigr)^{1-2/p}\\
&\leq \Bigl(1-\frac{2}{p} \Bigr)^{1-2/p}
\Bigl(\int_{\mathbb H}|({\rm Im}\,w)\Phi'(w)|^p\frac{dudv}{|{\rm Im}\,w|^2} \Bigr)^{2/p}.
\end{align}
Taking the supremum over $I$ yields $\Vert \Phi \Vert_{\rm BMOA} \leq c'_p\,\Vert \Phi \Vert_{{\mathcal B}_p}$ with 
$c_p'=(1-\frac{2}{p})^{1/2-1/p}$. For $1<p \leq 2$, combine this with (i).
\end{proof}

Arguing as in (i) above with the definition
$\Vert \Phi \Vert_{{\mathcal B}_{\infty}^{\#}}=\sup_{z \in \mathbb H}|({\rm Im}\,z)^2\Phi''(z)|$
yields:

\begin{proposition}\label{norm2}
For $1 \leq p \leq q \leq \infty$, 
there exists a constant $\tilde c_{p,q}>0$ such that
$\Vert \Phi\Vert_{{\mathcal B}_q^{\#}} \leq \tilde c_{p,q}\,\Vert \Phi\Vert_{{\mathcal B}_p^{\#}}$.
\end{proposition}

\begin{definition}
For $1 \leq p <\infty$, set
$$
\Vert \Phi \Vert_{\widehat{\mathcal B}_p}=\Vert \Phi \Vert_{{\mathcal B}_p^{\#}}+\Vert \Phi\Vert_{\rm BMOA}.
$$
The collection of $\Phi$ with 
$\Vert \Phi\Vert_{\widehat{\mathcal B}_p}<\infty$
is denoted by $\widehat{\mathcal B}_p(\mathbb H)$; equivalently,
$$
\widehat{\mathcal B}_p(\mathbb H)={\mathcal B}_p^{\#}(\mathbb H) \cap {\rm BMOA}(\mathbb H).
$$
\end{definition}

Hereafter, to suppress multiplicative constants, we use the notation
$A(\tau) \lesssim B(\tau)$ to mean that there exists $C>0$
such that $A(\tau) \leq C\,B(\tau)$ uniformly in the relevant parameter $\tau$;
we write $A(\tau) \asymp B(\tau)$ when both $A(\tau) \lesssim B(\tau)$ and $A(\tau) \gtrsim B(\tau)$ hold.

By Proposition \ref{norm2}, we have 
$\Vert \Phi \Vert_{\widehat{\mathcal B}_q} \lesssim \Vert \Phi \Vert_{\widehat{\mathcal B}_p}$
for $1 \leq p \leq q$. 

\begin{proposition}\label{new}
Let $1<p<\infty$. Then $\Vert \Phi \Vert_{{\mathcal B}_p^{\#}} \lesssim \Vert \Phi \Vert_{{\mathcal B}_p}$
for $\Phi \in {\mathcal B}_p(\mathbb H)$. 
Conversely, 
$\Vert \Phi \Vert_{{\mathcal B}_p} \lesssim \Vert \Phi \Vert_{{\mathcal B}_p^{\#}}+\Vert \Phi \Vert_{{\mathcal B}_\infty}$ for
$\Phi \in {\mathcal B}_p^{\#}(\mathbb H) \cap {\mathcal B}_\infty(\mathbb H)$.
Hence, the seminorms $\Vert \Phi \Vert_{{\mathcal B}_p}$ and $\Vert \Phi \Vert_{\widehat{\mathcal B}_p}$ are equivalent.
\end{proposition}

\begin{proof}
For the first inequality, we adapt the proof of \cite[Lemma~3.3]{STW}.
By the Cauchy integral formula,
$$
|\Phi''(z)| \leq \frac{1}{2\pi}\int_{|\zeta-z|=y/4} \frac{|\Phi'(\zeta)|}{|\zeta-z|^2}|d\zeta|
\leq \frac{4}{y} \max_{|\zeta-z|\leq y/4} |\Phi'(\zeta)|
$$
for $z=x+iy \in \mathbb H^+$. 
Moreover,
$$
|\Phi'(\zeta)|^p \leq 
\frac{16}{\pi y^2} \int_{|w-\zeta| \leq y/4}|\Phi'(w)|^p\,du\,dv
$$
for $w=u+iv$. Hence
$$
y^{2p-2}|\Phi''(z)|^p \lesssim y^{p-4} \int_{|w-z|\leq y/2} |\Phi'(w)|^p\,du\,dv
\leq y^{p-4} \int_{y/2}^{3y/2} \int_{x-y/2}^{x+y/2} |\Phi'(w)|^p\,du\,dv.
$$
With the change of variables $(u,v) \mapsto (\xi,\eta)$ by $u=x+y\xi$ and $v=y\eta$,
the right-hand side becomes
$$
y^{p-2} \int_{1/2}^{3/2} \int_{-1/2}^{1/2} |\Phi'(x+y\xi+iy \eta)|^p \,d\xi\,d\eta.
$$

Using the inequality of this form, we estimate $\Vert \Phi \Vert_{{\mathcal B}_p^{\#}}$ as
\begin{align*}
\Vert \Phi \Vert_{{\mathcal B}_p^{\#}}^p&=\int_{\mathbb H} y^{2p-2}|\Phi''(z)|^pdxdy\\
&\lesssim  \int_{1/2}^{3/2} \int_{-1/2}^{1/2} \Bigl(\int_{\mathbb H} y^{p-2}|\Phi'(x+y\xi+iy \eta)|^p dxdy \Bigr) d\xi d\eta. 
\end{align*}
Again with the change of variables $(x,y) \mapsto (u,v)$ by $u=x+y\xi$ and $v=y \eta$, the last integral turns out to be
$$
\int_{1/2}^{3/2} \int_{-1/2}^{1/2} \Bigl(\int_{\mathbb H} \Bigl(\frac{v}{\eta} \Bigr)^{p-2}|\Phi'(w)|^p \frac{dudv}{\eta} \Bigr) d\xi d\eta
=\Bigl(\int_{1/2}^{3/2} \frac{d\eta}{\eta^{p-1}}\Bigr)\int_{\mathbb H} v^{p-2}|\Phi'(w)|^p dudv
\asymp \Vert \Phi \Vert_{{\mathcal B}_p}^p.
$$
Thus, $\Vert \Phi \Vert_{{\mathcal B}_p^{\#}} \lesssim \Vert \Phi \Vert_{{\mathcal B}_p}$ is verified.

For the converse, \cite[Lemma~3.2]{LS} essentially gives
$$
\Vert \Phi \Vert_{{\mathcal B}_p} \lesssim \Vert \Phi \Vert_{{\mathcal B}_p^{\#}}+
\Vert \Phi'' \Vert_{{\mathcal A}_\infty},
$$
where $\Vert \Phi'' \Vert_{{\mathcal A}_\infty}$ is defined later in \eqref{Ap} and satisfies
$\Vert \Phi'' \Vert_{{\mathcal A}_\infty} \asymp \Vert \Phi \Vert_{{\mathcal B}_\infty}$ by
\cite[Lemma~6.3]{ShenV}. This yields the stated bound.
\end{proof}

\begin{remark}
In the second statement, for $p>2$ one even has
$\Vert \Phi \Vert_{{\mathcal B}_p} \lesssim \Vert \Phi \Vert_{{\mathcal B}_p^{\#}}$ for
$\Phi \in {\mathcal B}_p^{\#}(\mathbb H) \cap {\mathcal B}_\infty(\mathbb H)$. Indeed,
from
$\Phi'(x+iy)=-i\int_{y}^{y_0} \Phi''(x+it)\,dt+\Phi'(x+iy_0)$ 
and letting $y_0 \to \infty$, we obtain
\begin{equation}\label{int-infty}
\Phi'(x+iy)=-i\int_{y}^{\infty} \Phi''(x+it)\,dt
\end{equation} 
for $x+iy \in \mathbb H^+$
since $\lim_{y_0 \to \infty} \Phi'(x+iy_0)=0$ when $\Phi \in \mathcal B_\infty(\mathbb H)$.
For $p>2$ and $1<q<2$ with $1/p+1/q=1$, this gives
\begin{align*}
|\Phi'(x+iy)| 
&\leq\Bigl(\int_y^\infty \frac{dt}{t^{2-q/p}} \Bigr)^{1/q}\Bigl(\int_y^\infty  t^{2p-3}|\Phi''(x+it)|^p\,dt\Bigr)^{1/p},
\end{align*}
hence 
$$
y^{p-2}|\Phi'(x+iy)|^p \lesssim \int_y^\infty  t^{2p-3}|\Phi''(x+it)|^p\,dt,
$$
and integrating over $\mathbb H$ and exchanging the order of integrals yield $\Vert \Phi \Vert_{{\mathcal B}_p}\lesssim \Vert \Phi \Vert_{{\mathcal B}_p^{\#}}$.
\end{remark}

Identifying functions that differ by a constant, we may regard ${\mathcal B}_p(\mathbb H)$ and $\widehat{\mathcal B}_p(\mathbb H)$ 
as normed spaces with norms $\Vert \cdot\Vert_{{\mathcal B}_p}$ and $\Vert \cdot\Vert_{\widehat{\mathcal B}_p}$,
respectively; under these norms they are complex Banach spaces.

For the unit disk $\mathbb D$, define ${\mathcal B}_p(\mathbb D)$, ${\mathcal B}_p^{\#}(\mathbb D)$, ${\rm BMOA}(\mathbb D)$,
and $\widehat{\mathcal B}_p(\mathbb D)$ analogously by replacing the
hyperbolic density $1/|{\rm Im}\,z|$ on $\mathbb H$ with $2/(1-|z|^2)$ on $\mathbb D$. 
Let $K(z)=(z-i)/(z+i)$ be the Cayley transformation, which maps
$\mathbb H^+$ conformally onto $\mathbb D$ with $K(i)=0$
(and $K(z)=(-z-i)/(-z+i)$ maps $\mathbb H^-$ onto $\mathbb D$ with $K(-i)=0$).
For a function $\Phi$ on $\mathbb H$, write $K_*(\Phi)=\Phi \circ K^{-1}$ for the push-forward to $\mathbb D$.
Then $K_*$ is an isometric isomorphism from
${\mathcal B}_p(\mathbb H)$ onto ${\mathcal B}_p(\mathbb D)$ for $p>1$ (including $p=\infty$),
by conformal invariance.
For the spaces $\widehat{\mathcal B}_p(\mathbb H)$ and $\widehat{\mathcal B}_p(\mathbb D)$,
which involve $\Phi''$, the situation is subtler.

To show that $K_*$ gives a Banach isomorphism between 
$\widehat{\mathcal B}_p(\mathbb H)$ and $\widehat{\mathcal B}_p(\mathbb D)$, we prepare the following lemma.
For a holomorphic function $\Phi$ on $\mathbb H^+$, the seminorm defined by
its derivative in the Hardy space $\mathcal H_1$ is 
\begin{equation}\label{h1H}
\Vert \Phi \Vert_{\dot {\mathcal H}_1^1}=\sup_{y>0}\int_{-\infty}^{\infty} |\Phi'(x+iy)|\,dx,
\end{equation}
and write $\dot {\mathcal H}_1^1(\mathbb H)$ for the corresponding space.
Similarly, for a holomorphic function $\Phi_*$ on $\mathbb D$, set
\begin{equation}\label{h1D}
\Vert \Phi_* \Vert_{\dot {\mathcal H}_1^1}=\sup_{0<r<1} \frac{1}{2\pi}\int_{0}^{2\pi} |\Phi_*'(re^{i\theta})|\,d\theta,
\end{equation}
and denote the space by $\dot {\mathcal H}_1^1(\mathbb D)$. 

\begin{lemma}\label{hardy}
{\rm (i)} 
Every $\Phi \in \mathcal B_\infty(\mathbb H)$ satisfies
$\Vert \Phi \Vert_{\dot {\mathcal H}_1^1} \leq \Vert \Phi \Vert_{{\mathcal B}_1^{\#}}$.
{\rm (ii)} Every holomorphic function $\Phi_*$ on $\mathbb D$ satisfies
$\Vert \Phi_* \Vert_{\dot {\mathcal H}_1^1} \leq C
\bigl(\Vert \Phi_* \Vert_{\mathcal B_1^{\#}}+\Vert \Phi_* \Vert_{{\mathcal B}_\infty}\bigr)$ for some absolute constant $C>0$.
\end{lemma}

\begin{proof}
(i) From \eqref{int-infty},
$$
\int_{-\infty}^\infty |\Phi'(x+iy)|\,dx \leq \int_{-\infty}^\infty\int_{y}^{\infty} |\Phi''(x+it)|\,dt\,dx
\leq \int_{\mathbb H} |\Phi''(z)|\,dx\,dy,
$$
and taking the supremum over $y>0$ gives the claim.

(ii) Likewise, $\Phi_*'(re^{i\theta})=\int_{\varepsilon}^r \Phi_*''(te^{i\theta})\,dt+\Phi_*'(\varepsilon e^{i\theta})$ for $0<\varepsilon<r<1$.
Hence
\begin{align}
\frac{1}{2\pi}\int_0^{2\pi} |\Phi_*'(re^{i\theta})| \,d\theta 
&\leq \frac{1}{2\pi}\int_0^{2\pi}
\int_{\varepsilon}^r |\Phi_*''(te^{i\theta})|\,dt\,d\theta+\frac{2}{1-\varepsilon^2}\Vert \Phi_* \Vert_{{\mathcal B}_\infty}\\
&\leq \frac{1}{2\pi \varepsilon}\Vert \Phi_* \Vert_{{\mathcal B}_1^{\#}}+\frac{2}{1-\varepsilon^2}\Vert \Phi_* \Vert_{{\mathcal B}_\infty},
\end{align}
which implies the claim with $C=\min_{0<\varepsilon <1} \max\{\tfrac{1}{2\pi \varepsilon},\tfrac{2}{1-\varepsilon^2}\}$.
\end{proof}

We can now establish the expected correspondence between $\widehat{\mathcal B}_p(\mathbb H)$ and 
$\widehat{\mathcal B}_p(\mathbb D)$. An idea for its proof is in \cite[Section 9]{Pe}.

\begin{theorem}\label{isomorphism}
The push-forward $K_*$ by the Cayley transformation is a Banach isomorphism from $\widehat{\mathcal B}_p(\mathbb H)$ onto 
$\widehat{\mathcal B}_p(\mathbb D)$ for $p \geq 1$. 
\end{theorem}

\begin{proof}
First, by conformal invariance, $\Vert K_*(\Phi) \Vert_{\rm BMOA} \asymp \Vert \Phi \Vert_{\rm BMOA}$:
BMO functions on $\mathbb R$ and $\mathbb S$ correspond under the Cayley transformation
(see \cite[Corollary.~VI.1.3]{Ga}), and BMOA functions are holomorphic functions obtained by the Poisson integral of those functions.

We estimate the ${\mathcal B}_p^{\#}$-seminorms.
Let $\Phi_*=K_*(\Phi)=\Phi \circ K^{-1}$. Changing variables $\zeta=K(z)$ gives
\begin{align}\label{formula1}
& \quad \int_{\mathbb H} |({\rm Im}\,z)^2\, \Phi''(z)|^p \frac{dx\,dy}{|{\rm Im}\,z|^2} \\
&=\int_{\mathbb H} |({\rm Im}\,z)^2\,(\Phi_*''\circ K(z)\cdot K'(z)^2+\Phi_*'\circ K(z)\cdot K''(z))|^p \frac{dx\,dy}{|{\rm Im}\,z|^2} \notag \\
& \leq 2^{p-1}\int_{\mathbb D} \Bigl|\Bigl(\frac{1-|\zeta|^2}{2}\Bigr)^2\Phi_*''(\zeta)\Bigr|^p\frac{4\,d\xi\,d\eta}{(1-|\zeta|^2)^2} \notag\\
&\quad +2^{p-1}\int_{\mathbb D}\Bigl|\Bigl(\frac{1-|\zeta|^2}{2}\Bigr)\Phi_*'(\zeta)\Bigr|^p \Bigl(\frac{1-|\zeta|^2}{|1-\zeta|}\Bigr)^p
\frac{4\,d\xi\,d\eta}{(1-|\zeta|^2)^2}.
\end{align}
Note that $1-|\zeta|^2 \leq 2|1-\zeta|$.

Suppose $\Phi_*\in \widehat{\mathcal B}_p(\mathbb D)$. For $p>1$, \eqref{formula1} yields
$$
\Vert \Phi \Vert_{{\mathcal B}_p^{\#}}^p \leq 2^{p-1}\Vert \Phi_* \Vert_{{\mathcal B}_p^{\#}}^p
+2^{2p-1}\Vert \Phi_* \Vert_{{\mathcal B}_p}^p.
$$
Because $\Vert \Phi_* \Vert_{{\mathcal B}_p} \asymp \Vert \Phi_* \Vert_{{\mathcal B}_p^{\#}}+|\Phi_*'(0)| \lesssim \Vert \Phi_* \Vert_{\widehat{\mathcal B}_p}$ (see \cite[p.~327]{Zhu0}),
we obtain $\Vert \Phi \Vert_{\widehat{\mathcal B}_p} \lesssim \Vert K_*(\Phi) \Vert_{\widehat{\mathcal B}_p}$.
When $p=1$, the second integral in \eqref{formula1} becomes
\begin{equation}\label{formula2}
\int_{\mathbb D}|\Phi_*'(\zeta)| 
\frac{2\,d\xi\,d\eta}{|1-\zeta|}. 
\end{equation}
By Lemma~\ref{hardy}, $\Phi_*\in \dot{\mathcal H}^1_1(\mathbb D)$. Moreover,
$dm^*=2\,d\xi\,d\eta/|1-\zeta|$ is a Carleson measure on $\mathbb D$.
This can be verified by straightforward computation;
indeed, it suffices to show that for a disk $\Delta(1,r)$ with center at $1$ and radius $r>0$, 
$$
\frac{1}{r} \int_{\Delta(1,r)}\frac{d\xi d\eta}{|1-\zeta|} \leq \pi.
$$

Here, we apply the Carleson embedding theorem (see \cite[Theorem 9.3]{Du}, \cite[Theorem II.3.9]{Ga}).
This in particular implies that for any holomorphic function $\Psi$ in the Hardy space ${\mathcal H}_1(\mathbb D)$ 
with norm $\Vert \cdot \Vert_{\mathcal H_1}$ 
and for any Carleson measure $dm^*$ on $\mathbb D$, there exists a constant $c'>0$ depending only on $dm^*$ such that
$\int_{\mathbb D} |\Psi(\zeta)|dm^*(\zeta) \leq c' \Vert \Psi \Vert_{\mathcal H_1}$. Thus,
integral \eqref{formula2} is bounded by $c' \Vert \Phi_* \Vert_{\dot{\mathcal H}^1_1}$.
Plugging this estimate into inequality \eqref{formula1} and using Lemma \ref{hardy}, we obtain that
$$
\Vert \Phi \Vert_{{\mathcal B}_1^{\#}} \leq \Vert K_*(\Phi) \Vert_{{\mathcal B}_1^{\#}}+c'\Vert K_*(\Phi) \Vert_{\dot{\mathcal H}^1_1}
\lesssim \Vert K_*(\Phi) \Vert_{\widehat{\mathcal B}_1}.
$$
This yields $\Vert \Phi \Vert_{\widehat{\mathcal B}_1} \lesssim \Vert K_*(\Phi) \Vert_{\widehat{\mathcal B}_1}$.

Conversely, assume $\Phi \in \widehat{\mathcal B}_p(\mathbb H)$. Likewise to the above computation, we have
\begin{align}\label{formula3}
& \quad \int_{\mathbb D} \Bigl|\Bigl(\frac{1-|\zeta|^2}{2}\Bigr)^2\Phi_*''(\zeta)\Bigr|^p \frac{4\,d\xi\,d\eta}{(1-|\zeta|^2)^2} \notag \\
& \leq 2^{p-1}\int_{\mathbb H} |({\rm Im}\,z)^2\,\Phi''(z)|^p\frac{dx\,dy}{|{\rm Im}\,z|^2}
+2^{p-1}\int_{\mathbb H}|({\rm Im}\,z)\,\Phi'(z)|^p \Bigl(\frac{2\,{\rm Im}\,z}{|z+i|}\Bigr)^{p}\frac{dx\,dy}{|{\rm Im}\,z|^2},
\end{align}
where ${\rm Im}\,z \le |z+i|$. For $p>1$, \eqref{formula3} implies
$$
\Vert K_*(\Phi) \Vert_{{\mathcal B}_p^{\#}}^p \leq 2^{p-1}\Vert \Phi \Vert_{{\mathcal B}_p^{\#}}^p
+2^{2p-1}\Vert \Phi \Vert_{{\mathcal B}_p}^p,
$$
and with Proposition~\ref{new} we get $\Vert K_*(\Phi) \Vert_{\widehat{\mathcal B}_p} \lesssim \Vert \Phi \Vert_{\widehat{\mathcal B}_p}$.
When $p=1$, the second integral on the right of \eqref{formula3} equals
\begin{equation}\label{formula4}
\int_{\mathbb H}|\Phi'(z)| 
\frac{2\,dx\,dy}{|z+i|}.  
\end{equation}
Since $\Phi \in \dot {\mathcal H}_1^1(\mathbb H)$ by Lemma~\ref{hardy} and $dm=2\,dx\,dy/|z+i|$ is a Carleson measure on $\mathbb H$, the Carleson embedding theorem implies that
\eqref{formula4} is bounded by $c''\Vert \Phi \Vert_{\dot{\mathcal H}^1_1}$ where $c''>0$ depends only on $dm$.
Using Lemma \ref{hardy} again for this, we obtain from \eqref{formula3} that
\begin{equation}\label{formula5}
\Vert K_*(\Phi) \Vert_{{\mathcal B}_1^{\#}} \leq \Vert \Phi \Vert_{{\mathcal B}_1^{\#}}
+c''\Vert \Phi \Vert_{\dot{\mathcal H}^1_1} \leq (1+c'')\Vert \Phi \Vert_{{\mathcal B}_1^{\#}},
\end{equation}
which implies $\Vert K_*(\Phi)  \Vert_{\widehat{\mathcal B}_1} \lesssim \Vert \Phi \Vert_{\widehat{\mathcal B}_1}$.
\end{proof}

We conclude the section with equivalent norms for $\Vert \cdot \Vert_{\widehat{\mathcal B}_1}$ on ${\widehat{\mathcal B}_1}(\mathbb H)$ and ${\widehat{\mathcal B}_1}(\mathbb D)$.

\begin{proposition}\label{H1}
{\rm (i)} On $\widehat{\mathcal B}_1(\mathbb D)$ the norm $\Vert \Phi_* \Vert_{\widehat{\mathcal B}_1}$ is equivalent to
$\Vert \Phi_* \Vert_{{\mathcal B}_1^{\#}}
+\Vert \Phi_* \Vert_{{\mathcal B}_\infty}$ and to
$\Vert \Phi_* \Vert_{{\mathcal B}^{\#}_1}+\Vert \Phi_* \Vert_{\dot {\mathcal H}_1^1}$.
{\rm (ii)} On $\widehat{\mathcal B}_1(\mathbb H)$ the norm $\Vert \Phi \Vert_{\widehat{\mathcal B}_1}$ is equivalent to
$\Vert \Phi \Vert_{{\mathcal B}_1^{\#}}
+\Vert \Phi \Vert_{{\mathcal B}_\infty}$ and to
$\Vert \Phi \Vert_{{\mathcal B}_1^{\#}}+\Vert \Phi \Vert_{\dot {\mathcal H}_1^1}$.
\end{proposition}

\begin{proof} 
(i) By using the facts that
$\Vert \Phi_* \Vert_{{\mathcal B}_p^{\#}} \lesssim \Vert \Phi_* \Vert_{{\mathcal B}_1^{\#}}$
for any $p>1$ (which is the same as Proposition \ref{norm2})
and $|\Phi_*'(0)| \leq \Vert \Phi_* \Vert_{\dot {\mathcal H}_1^1}$, we obtain 
\begin{equation}\label{formula6}
\Vert \Phi_* \Vert_{\rm BMOA} \lesssim
\Vert \Phi_* \Vert_{{\mathcal B}_p} \asymp
\Vert \Phi_* \Vert_{{\mathcal B}_p^{\#}}+|\Phi_*'(0)| \lesssim
\Vert \Phi_* \Vert_{{\mathcal B}_1^{\#}}+\Vert \Phi_* \Vert_{\dot {\mathcal H}_1^1}. 
\end{equation}
Hence 
$\Vert \Phi_* \Vert_{\widehat{\mathcal B}_1} \lesssim
\Vert \Phi_* \Vert_{{\mathcal B}_1^{\#}}+\Vert \Phi_* \Vert_{\dot {\mathcal H}_1^1}$.
The bound $\Vert \Phi_* \Vert_{{\mathcal B}_1^{\#}}+\Vert \Phi_* \Vert_{{\mathcal B}_\infty} 
\lesssim \Vert \Phi_* \Vert_{\widehat{\mathcal B}_1}$ is immediate.
Finally, $\Vert \Phi_* \Vert_{{\mathcal B}_1^{\#}}+\Vert \Phi_* \Vert_{\dot {\mathcal H}_1^1}
\lesssim \Vert \Phi_* \Vert_{{\mathcal B}_1^{\#}}+\Vert \Phi_* \Vert_{{\mathcal B}_\infty}$ follows from Lemma \ref{hardy}.

(ii) Transfer the estimate for $\Phi_*= K_*(\Phi)$ on $\mathbb D$ back to $\Phi$ on $\mathbb H$.
From \eqref{formula5}, 
$\Vert K_*(\Phi) \Vert_{{\mathcal B}_1^{\#}} \lesssim \Vert \Phi \Vert_{{\mathcal B}_1^{\#}}$.
Moreover, $\Vert K_*(\Phi)\Vert_{\dot {\mathcal H}_1^1} \lesssim 
\Vert \Phi \Vert_{\dot {\mathcal H}_1^1}$. Indeed, the line integral along the horizontal line in \eqref{h1H}
is transfered by $K$ to the line integral along a horocycle in $\mathbb D$ tangent at $1$, which dominates
the integral along the circle in \eqref{h1D}; see the argument in \cite[Section 11.1]{Du}. 
Thus, by \eqref{formula6},
$$
\Vert \Phi \Vert_{\rm BMOA} \asymp \Vert K_*(\Phi) \Vert_{\rm BMOA} 
\lesssim \Vert K_*(\Phi) \Vert_{{\mathcal B}_1^{\#}}+\Vert K_*(\Phi) \Vert_{\dot {\mathcal H}_1^1}
\lesssim \Vert \Phi \Vert_{{\mathcal B}_1^{\#}}+\Vert \Phi \Vert_{\dot {\mathcal H}_1^1},
$$
which implies 
$\Vert \Phi \Vert_{\widehat{\mathcal B}_1} \lesssim
\Vert \Phi \Vert_{{\mathcal B}_1^{\#}}+\Vert \Phi \Vert_{\dot {\mathcal H}_1^1}$.
The remaining implications are as in (i), again using Lemma \ref{hardy}.
\end{proof}

\begin{remark}\label{normHD}
Proposition \ref{H1} implies in particular that 
$$
\widehat{\mathcal B}_1(\mathbb H)={\mathcal B}_1^{\#}(\mathbb H) \cap {\mathcal B}_\infty(\mathbb H)
={\mathcal B}_1^{\#}(\mathbb H) \cap \dot {\mathcal H}_1^1(\mathbb H),
$$
and in fact ${\mathcal B}_1^{\#}(\mathbb H) \nsubseteq {\mathcal B}_\infty(\mathbb H)$ and
${\mathcal B}_1^{\#}(\mathbb H) \nsubseteq \dot{\mathcal H}_1^1(\mathbb H)$.
By contrast, ${\mathcal B}_1^{\#}(\mathbb D) \subset {\mathcal B}_\infty(\mathbb D)$ since
every function in ${\mathcal B}_1^{\#}(\mathbb D)$ is bounded (see \cite[Theorem 5.19]{Zhu}), hence
${\mathcal B}_1^{\#}(\mathbb D) \subset \dot{\mathcal H}_1^1(\mathbb D)$ 
by Lemma \ref{hardy}. Consequently, ${\mathcal B}_1^{\#}(\mathbb D)=\widehat {\mathcal B}_1(\mathbb D)$ by
Proposition \ref{H1}.
\end{remark}

\begin{remark}\label{normsupply}
In defining $\widehat{\mathcal B}_p(\mathbb H)$ we included the BMOA seminorm $\Vert \Phi \Vert_{\rm BMOA}$,
but one could equally well use $\Vert \Phi \Vert_{\mathcal B_\infty}$ or $\Vert \Phi \Vert_{\dot{\mathcal H}_1^1}$.
The specific choice is not essential; our goals are twofold:
(1) to ensure that the seminorm on $\widehat{\mathcal B}_p(\mathbb H)$ annihilates only constants, and
(2) to preserve the Banach isomorphism between 
$\widehat{\mathcal B}_p(\mathbb H)$ and $\widehat{\mathcal B}_p(\mathbb D)$ 
under the Cayley transformation.
\end{remark}

\section{The pre-Schwarzian derivative map}\label{3}

We consider the properties of conformal mappings induced by integrable Beltrami coefficients.
A measurable function $\mu$ on $\mathbb H$ with $\Vert \mu \Vert_\infty<1$ is called 
a {\it Beltrami coefficient}. The set of all Beltrami coefficients on $\mathbb H$ is denoted by
$M(\mathbb H)$, which is the open unit ball of $L_\infty(\mathbb H)$ with 
the supremum norm $\Vert \mu \Vert_\infty$. 

\begin{definition}
For $p \geq 1$, the space of $p$-integrable
Beltrami coefficients is defined by
$$
M_p(\mathbb H)=\Bigl\{\mu \in M(\mathbb H) \mid 
\Vert \mu \Vert_p=\Bigl(\int_{\mathbb H} |\mu(z)|^p\frac{dxdy}{|{\rm Im}\,z|^2}\Bigr)^{1/p}<\infty\Bigr\}.
$$ 
We equip $M_p(\mathbb H)$ with the norm $\Vert \mu \Vert_p+\Vert \mu \Vert_\infty$.
\end{definition}

For $\mu \in M_p(\mathbb H^+)$, we denote by $F^\mu$ the normalized conformal homeomorphism of $\mathbb H^-$ that extends quasiconformally to $\mathbb C$ with complex dilatation $\mu$ on $\mathbb H^+$. 
The normalization is given by fixing
the three points $0$, $1$, and $\infty$. For a conformal homeomorphism $F:\mathbb H^- \to \mathbb C$,
the pre-Schwarzian derivative $N_F$ and the Schwarzian derivative $S_F$ are defined by
$$
N_F=(\log F')'\ ; \quad S_F=(N_F)'-\tfrac12(N_F)^2.
$$
For the conformal homeomorphism $F^\mu$ of $\mathbb H^-$ with $\mu \in M(\mathbb H^+)$,
let $L(\mu)=\log (F^\mu)'$ and $S(\mu)=S_{F^\mu}$. We call the maps $L$ and $S$ on $M(\mathbb H^+)$
the {\it pre-Schwarzian} and the {\it Schwarzian derivative maps}.

For $p \geq 1$, we define the norm
\begin{equation}\label{Ap}
\Vert \Phi\Vert_{{\mathcal A}_p}=\Bigl(\int_{\mathbb H} |({\rm Im}\,z)^2\, \Phi(z)|^p \frac{dxdy}{|{\rm Im}\,z|^2}\Bigr)^{1/p}
\end{equation}
for holomorphic functions $\Phi$ on $\mathbb H$. For $p=\infty$,
we set $\Vert \Phi\Vert_{{\mathcal A}_\infty}=\sup_{z \in \mathbb H}|({\rm Im}\,z)^2\, \Phi(z)|$.
The set of all such $\Phi$ with $\Vert \Phi\Vert_{{\mathcal A}_p}<\infty$
is denoted by ${\mathcal A}_p(\mathbb H)$, which is a complex Banach space with this norm. 

The Schwarzian derivative map $S$ on $M_p(\mathbb H^+)$ has been thoroughly studied. In addition,
we obtain the following result; see \cite[Lemma 3.2]{WM-1}. 
Remark \ref{S-case} below outlines the proof of the first statement.

\begin{proposition}\label{S-holo}
For $p \geq 1$, there exists a constant $\widetilde C_p>0$ such that
the Schwarzian derivative map $S$ satisfies $\Vert S(\mu)\Vert_{\mathcal A_p} \leq \widetilde C_p \Vert \mu \Vert_p$
for every $\mu \in M_p(\mathbb H^+)$. Moreover,
$S:M_p(\mathbb H^+) \to {\mathcal A}_p(\mathbb H^-)$ 
is holomorphic.
\end{proposition}

We note that the holomorphy of $S$ follows from its local boundedness in this situation.
Since this result will be used repeatedly later, we state it in a more general form.
Let $X$ and $Y$ be complex Banach spaces and let $W \subset X$ be a domain.
A map $J:W \to Y$ is called {\it G\^ateaux holomorphic} if, for any $w \in W$ and $x \in X$,
the function $J(w+\zeta x)$ is holomorphic in $\zeta \in \mathbb C$ into $Y$ in some neighborhood of the origin.
The following result is known (see \cite[Theorem 14.9]{Ch}).

\begin{proposition}\label{G-holomorphic}
If $J:W \to Y$ is locally bounded and G\^ateaux holomorphic, then $J$ is holomorphic.
\end{proposition}

Moreover, when $Y$ is a complex Banach space of holomorphic functions, 
the G\^ateaux holomorphy can be verified in several ways; see \cite[Lemma V.5.1]{Le} and \cite[Lemma 6.1]{WM-2}.

We prove the same claim for the pre-Schwarzian derivative map $L$.
First, we show it under a special assumption on $p$. This is mentioned without proof
in the proof of \cite[Theorem 6.10]{WM-4}.

\begin{lemma}\label{P-bdd}
For $p>2$, there exists a constant $C_p>0$ depending only on $p$ such that 
the pre-Schwarzian derivative map $L$
satisfies $\Vert L(\mu) \Vert_{{\mathcal B}_p} \leq C_p \Vert \mu \Vert_p$ for
every $\mu \in M_p(\mathbb H^+)$. 
\end{lemma}

\begin{proof}
We first represent the directional derivative $d_\mu L(\nu)$ of $L$ at $\mu \in M_p(\mathbb H^+)$ in the direction of 
a tangent vector $\nu$. 
Let $\Omega^+=F(\mathbb H^+)$ and $\Omega^-=F(\mathbb H^-)$ for the quasiconformal extension $F$ of $F^\mu$
to $\mathbb C$, and let $\rho_+$ and $\rho_-$ denote
their hyperbolic densities. For the normalized Riemann mapping $G:\mathbb H^+ \to \Omega^+$
associated with $F$,
the push-forward of the Beltrami coefficient $\nu$ on $\mathbb H^+$ by $G$
is defined by
$$
G_*(\nu)(w)=\nu(G^{-1}(w))\frac{(G^{-1})_{\bar w}}{(G^{-1})_{w}} \quad (w \in \Omega^+).
$$
As in the case of the Schwarzian derivative map (see \cite[Lemma 5]{Har} and \cite[Theorem I.2.3]{TT}), we see that
\begin{equation}\label{derivative}
d_\mu L'(\nu)(F^{-1}(\zeta))(F^{-1})'(\zeta)=-\frac{2}{\pi} \int_{\Omega^+} \frac{G_*(\nu)(w)}{(w-\zeta)^3}dudv 
\quad (\zeta \in \Omega^-).
\end{equation}
Here, $d_\mu L'(\nu)$ stands for the derivative of the holomorphic function $d_\mu L(\nu)$ in ${\mathcal B}_p(\mathbb H^-)$.

We estimate the norm of $d_\mu L(\nu)$:
\begin{align}\label{d-norm}
\Vert d_\mu L(\nu) \Vert_{{\mathcal B}_p}^p
&=\int_{\mathbb H^-} |({\rm Im}\,z)\,d_\mu L'(\nu)(z)|^p\frac{dxdy}{|{\rm Im}\,z|^2}\\
&=\int_{\Omega^-} |d_\mu L'(\nu)(F^{-1}(\zeta))(F^{-1})'(\zeta)|^p \rho_-^{2-p}(\zeta) d\xi d\eta \\
&=\Bigl(\frac{2}{\pi}\Bigr)^p\int_{\Omega^-} \Bigl|\int_{\Omega^+} \frac{G_*(\nu)(w)}{(w-\zeta)^3}dudv \Bigr|^p 
\rho_-^{2-p}(\zeta) d\xi d\eta.
\end{align}
Then, applying the H\"older inequality to the absolute value of the inner integral, we obtain
\begin{align}\label{hoelder}
\Bigl|\int_{\Omega^+} \frac{G_*(\nu)(w)}{(w-\zeta)^3}dudv \Bigr|^p
\leq \Bigl(\int_{\Omega^+} \frac{1}{|w-\zeta|^{4-q}dudv}\Bigr)^{p/q}\Bigl(\int_{\Omega^+} \frac{|G_*(\nu)(w)|^p}{|w-\zeta|^4}dudv\Bigr)
\end{align}
for $1/p+1/q=1$.
Here, we note the following inequalities for the hyperbolic densities (see \cite[p.6]{Le}):
$$
\rho_-(\zeta) \geq \frac{1}{2d(\zeta, \partial \Omega^-)};\quad \rho_+(w) \geq \frac{1}{2d(w, \partial \Omega^+)}.
$$
Then, by virtue of the condition $q<2$, the first integral is bounded as follows:
\begin{align}\label{q<2}
\int_{\Omega^+} \frac{1}{|w-\zeta|^{4-q}}dudv 
&\leq \int_{|w-\zeta| \geq d(\zeta,\partial \Omega^-)}\frac{1}{|w-\zeta|^{4-q}}dudv \nonumber\\
&=\int_0^{2\pi}\!\! \int_{d(\zeta,\partial \Omega^-)}^\infty\frac{1}{r^{3-q}}drd\theta \nonumber\\
&=\frac{2\pi}{2-q}\frac{1}{d(\zeta,\partial \Omega^-)^{2-q}} \leq \frac{8\pi}{2-q}\rho_-^{2-q}(\zeta).
\end{align}
In the same way, we also have
\begin{align}\label{q=0}
\int_{\Omega^-} \frac{1}{|w-\zeta|^{4}}d\xi d\eta \leq 4\pi \rho_+^{2}(w).
\end{align}
The substitution of the above inequalities \eqref{hoelder}, \eqref{q<2}, and \eqref{q=0} into \eqref{d-norm} yields
\begin{align}\label{d-estimate}
\Vert d_\mu L(\nu) \Vert_{{\mathcal B}_p}^p
&\leq \Bigl(\frac{2}{\pi}\Bigr)^p \Bigl(\frac{8\pi}{2-q}\Bigr)^{p/q}\int_{\Omega^-} 
\int_{\Omega^+} \Bigl(\frac{|G_*(\nu)(w)|^p}{|w-\zeta|^4}dudv \Bigr) (\rho_-^{2-q}(\zeta))^{p/q} 
\rho_-^{2-p}(\zeta) d\xi d\eta \nonumber\\
&\leq \Bigl(\frac{16}{2-q}\Bigr)^p \int_{\Omega^+} \Bigl(\int_{\Omega^-} \frac{1}{|w-\zeta|^4}d\xi d\eta \Bigr)
|G_*(\nu)(w)|^p dudv \nonumber\\
&\leq \Bigl(\frac{16}{2-q}\Bigr)^p \int_{\Omega^+} 4\pi \rho_+^2(w)|G_*(\nu)(w)|^p dudv \nonumber\\
&=4\pi \Bigl(\frac{16}{2-q}\Bigr)^p \int_{\mathbb H^+} |\nu(z)|^p \frac{dxdy}{|{\rm Im}\,z|^2}=4\pi \Bigl(\frac{16}{2-q}\Bigr)^p \Vert \nu \Vert_p^p.
\end{align}

For $\mu \in M_p(\mathbb H^+)$, let $L_\mu(t)=L(t\mu)$ for $t \in [0,1]$. 
By the fundamental theorem of calculus, we have
$$
L(\mu)=L_\mu(1)-L_\mu(0)=\int_0^1 \frac{dL_\mu}{dt}(t) dt,
$$
where $\frac{dL_\mu}{dt}(t)=d_{t\mu}L(\mu)$. Inequality \eqref{d-estimate} proved above shows that
$$
\Vert d_{t\mu}L(\mu) \Vert_{{\mathcal B}_p}^p \leq C_p^p \Vert \mu \Vert_p^p
$$
for all $t \in [0,1]$, where $C_p>0$ is the constant depending only on $p$. Hence,
\begin{align}
\Vert L(\mu) \Vert_{{\mathcal B}_p}^p&=\int_{\mathbb H^-}
\Bigl|\Bigl(\int_0^1 \frac{dL_\mu}{dt}(t) dt\Bigr)'(z)\Bigr|^p
|{\rm Im}\,z|^{p-2} dxdy\\
&\leq \int_{\mathbb H^-}
\Bigl(\int_0^1 |d_{t\mu}L'(\mu)(z)| dt\Bigr)^p
|{\rm Im}\,z|^{p-2} dxdy\\
&\leq 
\int_0^1 \Bigl(\int_{\mathbb H^-} |d_{t\mu}L'(\mu)(z)|^p |{\rm Im}\,z|^{p-2} dxdy \Bigr) dt
=\Vert d_{t\mu}L(\mu) \Vert_{{\mathcal B}_p}^p,
\end{align}
which is bounded also by $C_p^p \Vert \mu \Vert_p^p$.
\end{proof}

\begin{remark}\label{S-case}
In the case of the Schwarzian derivative map $S:M_p(\mathbb H^+) \to {\mathcal A}_p(\mathbb H^-)$, 
a similar argument can be applied. This has been done in Theorem 2.3 and Lemma 2.9 of \cite[Chapter I]{TT}. 
The corresponding formula to \eqref{derivative} is
$$
d_\mu S(\nu)(F^{-1}(\zeta))(F^{-1})'(\zeta)^2=-\frac{6}{\pi} \int_{\Omega^+} \frac{G_*(\nu)(w)}{(w-\zeta)^4}dudv 
\quad (\zeta \in \Omega^-),
$$
and \eqref{hoelder} with the density $\rho_-^{2-2p}(\zeta)$ turns out to be
\begin{align}
\Bigl|\int_{\Omega^+} \frac{G_*(\nu)(w)}{(w-\zeta)^4}dudv \Bigr|^p\rho_-^{2-2p}(\zeta)
&\leq \Bigl(\int_{\Omega^+} \frac{1}{|w-\zeta|^{4}dudv}\Bigr)^{p/q}\Bigl(\int_{\Omega^+} \frac{|G_*(\nu)(w)|^p}{|w-\zeta|^4}dudv\Bigr)\rho_-^{2-2p}(\zeta)\\
& \leq (4\pi)^{p/q}\int_{\Omega^+} \frac{|G_*(\nu)(w)|^p}{|w-\zeta|^4}dudv
\end{align}
by using \eqref{q=0}. This holds without any condition on $p \geq 1$. In the case $p=1$,
the usual modification is applied for $q=\infty$. The other parts of the proof are the same.
This gives the first statement of Proposition \ref{S-holo}.
\end{remark}

We remove the condition $p>2$ in the statement of Lemma \ref{P-bdd} and show the required result
in full generality with the aid of properties of the Schwarzian derivative map $S$.

\begin{theorem}\label{P-holo}
For $p \geq 1$, 
the pre-Schwarzian derivative map $L$
satisfies $\Vert L(\mu) \Vert_{{\mathcal B}^{\#}_p} \leq C_p^{\#} \Vert \mu \Vert_p$ for
every $\mu \in M_p(\mathbb H^+)$, where $C_p^{\#}>0$ is a constant depending on $p$ and $\Vert \mu \Vert_p$.
Moreover,
$L:M_p(\mathbb H^+) \to \widehat{\mathcal B}_p(\mathbb H^-)$
is holomorphic.
\end{theorem}

\begin{proof}
For any $\mu \in M_p(\mathbb H^+)$, let $F=F^\mu$.
Then, using $S_F=(N_F)'-\tfrac12(N_F)^2$, we have
\begin{align}\label{formula7}
\Vert L(\mu) \Vert_{\mathcal B_p^{\#}}^p
&=\int_{\mathbb H^-}|({\rm Im}\,z)^{2}(N_{F})'(z)|^p\frac{dxdy}{|{\rm Im}\,z|^2} \nonumber \\
& \leq 2^{p-1}\int_{\mathbb H^-}|({\rm Im}\,z)^{2}S_{F}(z)|^p \frac{dxdy}{|{\rm Im}\,z|^2}
+\frac12\int_{\mathbb H^-}|({\rm Im}\,z) N_{F}(z)|^{2p}\frac{dxdy}{|{\rm Im}\,z|^2} \nonumber \\
& \leq
2^{p-1}\Vert S(\mu) \Vert_{\mathcal A_p}^p+
\tfrac12\Vert L(\mu) \Vert_{\mathcal B_{2p}}^{2p}.
\end{align}

We first assume $p>1$. By Proposition \ref{S-holo}, Lemma \ref{P-bdd}, and $\Vert \mu \Vert_{2p} \leq \Vert \mu \Vert_{p}$,
inequality \eqref{formula7} implies that
\begin{align}
\Vert L(\mu) \Vert_{{\mathcal B}_p^{\#}}^p &\leq 2^{p-1}(\widetilde C_p \Vert \mu \Vert_p)^p+\tfrac12(C_{2p}\Vert \mu \Vert_{2p})^{2p}\\
&\leq 2^{p-1}({\widetilde C_{p}}^{\ p}+C_{2p}^{\ 2p}\Vert \mu \Vert_{p}^p) \Vert \mu \Vert_{p}^p.
\end{align}
This yields $\Vert L(\mu) \Vert_{{\mathcal B}_p^{\#}} \leq C_p^{\#}\Vert \mu \Vert_{p}$ for $p>1$, where
$C_p^{\#}>0$ is a constant depending also on $\Vert \mu \Vert_{p}$.

In the case $p=1$, we apply \eqref{formula7} again to have
$$
\Vert L(\mu) \Vert_{\mathcal B_1^{\#}} \leq
\Vert S(\mu) \Vert_{\mathcal A_1}+
\Vert L(\mu) \Vert_{\mathcal B_{2}}^2.
$$ 
By using $\Vert \mu \Vert_2 \leq \Vert \mu \Vert_1$,
this implies that
$$
\Vert L(\mu) \Vert_{\mathcal B_1^{\#}} \leq \widetilde C_1 \Vert \mu \Vert_{1}+(C_2\Vert \mu \Vert_{1})^2.
$$
Hence, we can also find $C_1^{\#}>0$ depending on $\Vert \mu \Vert_{1}$ such that 
$\Vert L(\mu) \Vert_{\mathcal B_1^{\#}} \leq C_1^{\#} \Vert \mu \Vert_1$.
This completes the proof of the first statement of the theorem.

For the second statement, we note that
$L:M(\mathbb H^+) \to {\mathcal B}_\infty(\mathbb H^-)$ satisfies
$\Vert L(\mu) \Vert_{\mathcal B_\infty} \leq 3 \Vert \mu \Vert_\infty$
(see \cite[Proposition 5.3]{GH}).
Then, combined with the first statement and Remark \ref{normsupply}, this yields that
$$
\Vert L(\mu) \Vert_{\widehat{\mathcal B}_p} \asymp \Vert L(\mu) \Vert_{{\mathcal B}_p^{\#}}
+\Vert L(\mu) \Vert_{{\mathcal B}_\infty}
\leq \max\,\{C_p^{\#},3\}(\Vert \mu \Vert_p+\Vert \mu \Vert_\infty)
$$
for every $p \geq 1$.
Hence, $L:M_p(\mathbb H^+) \to \widehat{\mathcal B}_p(\mathbb H^-)$ is in particular locally bounded.
Under this condition, the standard argument implies that $L$ is in fact holomorphic
in virtue of Proposition \ref{G-holomorphic}.
\end{proof}

\begin{remark}
The continuity of $L:M_p(\mathbb H^+) \to \widehat{\mathcal B}_p(\mathbb H^-)$
can be proved directly as in \cite[Theorem 2.4]{Sh} and \cite[Theorem 2.4]{TS}, from which holomorphy also follows. Indeed,
for any $\mu, \nu \in M_p(\mathbb H)$, the same argument as above gives
\begin{equation}
\Vert L(\mu)-L(\nu) \Vert_{\mathcal B_p^{\#}}^p
\leq
2^{p-1}\{\Vert S(\mu)-S(\nu) \Vert_{\mathcal A_p}^p+
(\Vert L(\mu) \Vert_{\mathcal B_{2p}}^p+\Vert L(\nu) \Vert_{\mathcal B_{2p}}^p)\Vert L(\mu)-L(\nu) \Vert_{\mathcal B_{2p}}^p\}.
\end{equation}
\end{remark}

\begin{remark}
Theorem \ref{P-holo} improves the statement of \cite[Theorem 6.10]{WM-4} by replacing the assumption $p>2$
with $p \geq 1$.
\end{remark}

\begin{corollary}
For $p \geq 1$, 
the derivative of the pre-Schwarzian derivative map $L$ at the origin
satisfies $\Vert d_{0}L(\mu) \Vert_{{\mathcal B}^{\#}_p} \leq C_p^{\#} \Vert \mu \Vert_p$ for
every $\mu \in M_p(\mathbb H^+)$.
\end{corollary}

Next, we link $S$ and $L$ by the canonical holomorphic map 
$J:{\mathcal B}_\infty(\mathbb H) \to \mathcal A_\infty(\mathbb H)$ 
defined by $\Phi \mapsto \Phi''-(\Phi')^2/2$ for $\Phi \in {\mathcal B}_\infty(\mathbb H)$.

\begin{lemma}\label{BtoA}
For each $p \geq 1$, every 
$\Phi \in \widehat{\mathcal B}_p(\mathbb H)$ satisfies
$\Vert J(\Phi) \Vert_{{\mathcal A}_p} \leq c_p\Vert \Phi \Vert_{\widehat{\mathcal B}_p}$,
where $c_p>0$ is a constant depending on $p$ and $\Vert \Phi \Vert_{\widehat{\mathcal B}_p}$.
Moreover, $J$ is holomorphic on $\widehat{\mathcal B}_p(\mathbb H)$ with respect to
$\Vert \cdot \Vert_{\widehat{\mathcal B}_p}$.

\end{lemma}

\begin{proof}
We have
\begin{align}
\Vert J(\Phi) \Vert_{{\mathcal A}_p}^p
&=\int_{\mathbb H} |({\rm Im}\,z)^{2}(\Phi''(z)-\tfrac12\Phi'(z)^2)|^p\frac{dxdy}{|{\rm Im}\,z|^2}\\
&\leq 2^{p-1} \int_{\mathbb H} |({\rm Im}\,z)^{2}\Phi''(z)|^p\frac{dxdy}{|{\rm Im}\,z|^2}
+\frac12\int_{\mathbb H} |({\rm Im}\,z)\Phi'(z)|^{2p}\frac{dxdy}{|{\rm Im}\,z|^2}\\
&=2^{p-1}\Vert \Phi \Vert_{{\mathcal B}_p^{\#}}^p+\tfrac12\Vert \Phi \Vert_{{\mathcal B}_{2p}}^{2p}.
\end{align}
Since $\Vert \Phi \Vert_{{\mathcal B}_{2p}} \asymp \Vert \Phi \Vert_{\widehat{\mathcal B}_{2p}}
\lesssim \Vert \Phi \Vert_{\widehat{\mathcal B}_{p}}$ by Propositions \ref{norm2} and \ref{new}, this implies that 
$\Vert J(\Phi) \Vert_{{\mathcal A}_p} \leq c_p\Vert \Phi \Vert_{\widehat{\mathcal B}_p}$ for some
$c_p>0$, and in particular, $J$ is locally bounded.
It is easy to see that $J:\widehat{\mathcal B}_p(\mathbb H) \to {\mathcal A}_p(\mathbb H)$ is G\^{a}teaux holomorphic, 
and hence holomorphic by Proposition \ref{G-holomorphic}.
\end{proof}

We consider the holomorphic map $J$ on the image $L(M_p(\mathbb H^+))$ of the pre-Schwarzian derivative map.
We note that $J$ is injective on $L(M(\mathbb H^+))$. Since $F^\mu$ is normalized by
fixing $\infty$, it is determined by $\mu \in M(\mathbb H^+)$ up to post-composition by affine transformations of $\mathbb C$.
Therefore, 
for $\mu, \nu \in M(\mathbb H^+)$, $S_{F^{\mu}}=S_{F^{\nu}}$ if and only if $N_{F^{\mu}}=N_{F^{\nu}}$.
This shows the injectivity of $J$ on $L(M(\mathbb H^+))$, and hence on $L(M_p(\mathbb H^+))$.

The existence of a local holomorphic right inverse of the Schwarzian derivative map $S$ is
a crucial fact for the holomorphy of $J^{-1}$. 
The following claim has appeared in \cite[Theorem 4.1]{WM-1}.
Its proof omits the argument of approximating a given Schwarzian derivative 
by those extending holomorphically
to the boundary; however, this part can be verified by using \cite[Proposition 3]{Su}. 

\begin{proposition}\label{S-section}
Let $S:M_p(\mathbb H^+) \to {\mathcal A}_p(\mathbb H^-)$
be the Schwarzian derivative map for $p \geq 1$.
For each $\Psi_0$ in $S(M_p(\mathbb H^+))$, there exists a
neighborhood $V_{\Psi_0}$ of $\Psi_0$ in ${\mathcal A}_p(\mathbb H^-)$ 
and a holomorphic map $\sigma:V_{\Psi_0} \to M_p(\mathbb H^+)$ such that
$S \circ \sigma$ is the identity on $V_{\Psi_0}$.
\end{proposition}

In addition, because
the quasiconformal homeomorphism of $\mathbb H^+$ corresponding to $\Psi \in V_{\Psi_0}$ can be
explicitly represented by using a real-analytic quasiconformal reflection and by solving the Schwarzian differential equation, it is a real-analytic diffeomorphism.

\begin{proposition}\label{real-analytic}
For the local holomorphic right inverse $\sigma:V_{\Psi_0} \to M_p(\mathbb H^+)$ of $S$
given in Proposition \ref{S-section}, let $\mu=\sigma(\Psi)$ for any $\Psi \in V_{\Psi_0}$.
Then, the quasiconformal homeomorphism $\widetilde F^\mu$ of $\mathbb H^+$ with $\widetilde F^\mu(\infty)=\infty$
whose complex dilatation is $\mu$ is a real-analytic diffeomorphism.
\end{proposition}

\begin{proof}
For $\Psi_0 \in {\mathcal A}_p(\mathbb H^-)$, it is proved in \cite[Lemma 4.3]{WM-1} that there exists
$\nu \in M_p(\mathbb H^+)$ such that $S(\nu)=\Psi_0$ and $\widetilde F^\nu:\mathbb H^+ \to \Omega^+$ 
is a real-analytic bi-Lipschitz diffeomorphism with respect to the hyperbolic metrics on $\mathbb H^+$ and
its image domain $\Omega^+ \subset \mathbb C$. 
Its conformal extension is $F^{\nu}:\mathbb H^- \to \Omega^-={\mathbb C} \setminus \overline{\Omega^+}$. 
Then, the quasiconformal reflection $r:\Omega^+ \to \Omega^-$
with respect to $\partial \Omega^+=\partial \Omega^-$ is defined by 
$$
r(\zeta)=F^\nu\Bigl(\overline{(\widetilde F^{\nu})^{-1}(\zeta)}\Bigr) \quad (\zeta \in \Omega^+),
$$
which is a real-analytic bi-Lipschitz diffeomorphism.

For any $\Psi \in V_{\Psi_0}$, we consider the push-forward $F^\nu_*(\Psi)$ by the conformal
homeomorphism $F^\nu:\mathbb H^- \to \Omega^-$ and solve the differential equation
$2w''(z)+F^\nu_*(\Psi)(z)w(z)=0$ on $\Omega^-$.
Let $w_1$ and $w_2$ be linearly independent solutions so normalized that $w_1w_2'-w_2w_1'=1$.
Then, $S(w_1/w_2)=F^\nu_*(\Psi)$ on $\Omega^-$, 
and
the quasiconformal homeomorphism $\widetilde F^\mu$ of $\mathbb H^+$ whose complex dilatation is $\mu=\sigma(\Psi)$ is given by
the composition of $\widetilde F^\nu:\mathbb H^+ \to \Omega^+$ with
$$
\frac{w_1(r(\zeta))+(\zeta-r(\zeta))w_1'(r(\zeta))}{w_2(r(\zeta))+(\zeta-r(\zeta))w_2'(r(\zeta))},
$$
which is a quasiconformal real-analytic diffeomorphism of $\Omega^+$. We can prove this by \cite[Lemma 4]{Su},
together with its subsequent comment and remark. In particular, $\widetilde F^\mu$ is a real-analytic
diffeomorphism of $\mathbb H^+$.
\end{proof}

Concerning a global right inverse of the Schwarzian derivative map $S$, 
the following result is proved in \cite[Theorem 1.4]{WM-3} in the case $p>1$.

\begin{proposition}\label{global}
For $p>1$, there exists a real-analytic map $\Sigma:S(M_p(\mathbb H^+)) \to M_p(\mathbb H^+)$ such that
$S \circ \Sigma$ is the identity on $S(M_p(\mathbb H^+))$. Moreover, every $\mu \in M_p(\mathbb H^+)$ in the image of $\Sigma$
induces a quasiconformal real-analytic diffeomorphism $\widetilde F^\mu$ of $\mathbb H^+$. 
\end{proposition}

We are ready to prove the desired claim.

\begin{theorem}\label{SL}
For $p \geq 1$,
the holomorphic map 
$J:\widehat{\mathcal B}_p(\mathbb H^-) \to {\mathcal A}_p(\mathbb H^-)$ with
$J \circ L=S$ is a biholomorphic homeomorphism between $L(M_p(\mathbb H^+))$ and
$S(M_p(\mathbb H^+))$. 
\end{theorem}

\begin{proof}
Since $J \circ L=S$,
the restriction $J|_{L(M_p(\mathbb H^+))}$
of the holomorphic map $J:{\mathcal A}_p(\mathbb H^-) \to \widehat{\mathcal B}_p(\mathbb H^-)$ 
given in Lemma \ref{BtoA}
sends $L(M_p(\mathbb H^+))$ into $S(M_p(\mathbb H^+))$ injectively. Conversely, Proposition \ref{S-section}
shows that, for every 
$\Psi_0 \in S(M_p(\mathbb H^+))$, there is a local holomorphic map $\sigma:V_{\Psi_0} \to M_p(\mathbb H^+)$ 
such that
$S \circ \sigma$ is the identity on $V_{\Psi_0} \subset S(M_p(\mathbb H^+))$. Then,
$J \circ L\circ \sigma$ is the identity on $V_{\Psi_0}$, and hence $L\circ \sigma$ is a local holomorphic right
inverse of $J$. This implies that $J$ is a
biholomorphic homeomorphism of $L(M_p(\mathbb H^+))$ onto $S(M_p(\mathbb H^+))$.
\end{proof}

\begin{corollary}\label{L-section}
For each $\Phi_0$ in $L(M_p(\mathbb H^+))$ 
with $p \geq 1$, there exists a
neighborhood $U_{\Phi_0}$ of $\Phi_0$ in $\widehat {\mathcal B}_p(\mathbb H^-)$ 
and a holomorphic map $\tau:U_{\Phi_0} \to M_p(\mathbb H^+)$ such that
$L \circ \tau$ is the identity on $U_{\Phi_0}$.
\end{corollary}

\begin{proof}
Let $\Psi_0=J(\Phi_0)$. We choose $V_{\Psi_0}$ and $\sigma:V_{\Psi_0} \to M_p(\mathbb H^+)$ as in Proposition \ref{S-section}.
Then, $U_{\Phi_0}=J^{-1}(V_{\Psi_0})$ and $\tau=\sigma \circ J$ possess the required properties. 
\end{proof}

By Proposition \ref{global}, we also have that $\Sigma \circ J$ is a global real-analytic right inverse of 
the pre-Schwarzian derivative map $L:M_p(\mathbb H^+) \to \widehat {\mathcal B}_p(\mathbb H^-)$ for $p>1$.

As a by-product of the above arguments, we can also obtain a characterization of 
$p$-integrable Beltrami coefficients in terms of the pre-Schwarzian and Schwarzian derivative maps.
This has been given in the case $p>1$ by generalizing \cite[Theorem 4.4]{STW}.
We remark that the reasoning of $(3) \Rightarrow (1)$ in \cite[Theorem 7.1]{WM-4} should be read as given below.

\begin{theorem}\label{conformal}
Let $F:\mathbb H^- \to \mathbb C$ be a conformal map with $F(\infty)=\infty$ that extends to a quasiconformal
homeomorphism of $\mathbb C$.
Then, the following conditions are equivalent for every $p \geq 1$:
\begin{enumerate}
\item
$F$ extends quasiconformally to $\mathbb H^+$ so that its complex dilatation is in $M_p(\mathbb H^+)$;
\item
$\log F'$ belongs to $\widehat{\mathcal B}_p(\mathbb H^-)$;
\item
$S_F$ belongs to ${\mathcal A}_p(\mathbb H^-)$.
\end{enumerate}
\end{theorem}

\begin{proof}
The implication $(1) \Rightarrow (2)$ is obtained by Theorem \ref{P-holo}, and
$(2) \Rightarrow (3)$ by Lemma \ref{BtoA}. 
We may consider $(3) \Rightarrow (1)$ on the unit disk because Schwarzian derivatives are
invariant under M\"obius transformations.
We have to show that
$S_F \in {\mathcal A}_p(\mathbb D)$ implies that $F$ has the desired quasiconformal extension.
However, the same proof as in \cite[Theorem 2]{Cu}, 
relying on the local quasiconformal extension by Becker and Pommerenke \cite[Satz 4]{BP}, 
applies for $p \geq 1$.
\end{proof}

\section{Fiber spaces in the unit disk model}\label{4}

Let $S:M_p(\mathbb D^*) \to {\mathcal A}_p(\mathbb D)$
be the Schwarzian derivative map and $L:M_p(\mathbb D^*) \to \widehat{\mathcal B}_p(\mathbb D)$
the pre-Schwarzian derivative map for $p \geq 1$, defined in a similar way for $\mathbb D$ and 
$\mathbb D^*=\widehat{\mathbb C} \setminus \overline{\mathbb D}$.
Almost all statements in the previous section are also valid for these maps.
The exception occurs for the holomorphic map 
$J:\widehat {\mathcal B}_p(\mathbb D) \to \mathcal A_p(\mathbb D)$ with $J \circ L=S$.
In fact, $J$ maps $L(M_p(\mathbb D^*))$ onto $S(M_p(\mathbb D^*))$ surjectively but not
injectively.
While the statements up to Proposition \ref{global} in the previous section can be 
translated directly to this case, Theorem \ref{SL} requires a modification regarding the 
injectivity of $J:L(M_p(\mathbb D^*)) \to S(M_p(\mathbb D^*))$. 
We will examine the structure of this map more closely.

First, we give the precise definition of the pre-Schwarzian derivative map
$L:M(\mathbb D^*) \to {\mathcal B}_\infty(\mathbb D)$ in the present setting. 
We impose the following normalization on $F^\mu$. 
For $\mu \in M(\mathbb D^*)$, let $F^\mu$ be the conformal
homeomorphism of $\mathbb D$ onto a bounded domain in $\mathbb C$ with $F^\mu(0)=0$ and $(F^\mu)'(0)=1$
that extends to
a quasiconformal self-homeomorphism of ${\mathbb C}$ with complex dilatation $\mu$ on $\mathbb D^*$. 
We assume $F^\mu(\infty)=\infty$.
This normalization uniquely determines $F^\mu$ by $\mu$, and we use the same notation for its quasiconformal extension.
Later, its restriction to $\mathbb D^*$ is denoted by $\widetilde F^\mu$ to distinguish it from the
conformal mapping on $\mathbb D$.
Then the pre-Schwarzian derivative map $L$ is defined by $L(\mu)=\log (F^\mu)'$, 
which belongs to ${\mathcal B}_\infty(\mathbb D)$. 
If $\mu \in M_p(\mathbb D^*)$, then $L(\mu) \in \widehat{\mathcal B}_p(\mathbb D)$.

The fact that $J$ is not injective on $L(M_p(\mathbb D^*))$ is seen from
the following proposition, which can be verified easily (see \cite[Proposition 3.1]{Ma5}).

\begin{proposition}\label{affine}
{\rm (i)}
For $\mu, \nu \in M(\mathbb D^*)$, we have $S_{F^{\mu}}=S_{F^{\nu}}$ if and only if  
$F^{\mu}=W \circ F^{\nu}$ on $\mathbb D$
for some M\"obius transformation $W$ of $\widehat{\mathbb C}$ such that  
$W \circ F^{\nu}(\mathbb D)$ is a bounded domain in $\mathbb C$. Moreover, $N_{F^{\mu}}=N_{F^{\nu}}$ if and only if
$F^{\mu}=W \circ F^{\nu}$ on $\mathbb D$
for some affine transformation $W$ of $\mathbb C$.
{\rm (ii)}
For any $\nu \in M_p(\mathbb D^*)$ with $p \geq 1$ and any M\"obius transformation $W$ 
such that $W \circ F^{\nu}(\mathbb D)$ is a bounded domain,  
there exists some $\mu' \in M_p(\mathbb D^*)$ such that $N_{W \circ F^{\nu}}=N_{F^{\mu'}}$. 
\end{proposition}

Furthermore, the above variations of $F^{\nu}$ by such M\"obius transformations $W$ with
$W \circ F^{\nu}(\mathbb D)$ bounded
yield all $\Phi=\log(W \circ F^{\nu})'$ ($\Phi'=N_{W \circ F^{\nu}}$) in $\widehat{\mathcal B}_p (\mathbb D)$
with $J(\Phi)=S(\nu)$. This is a special case of the more general result shown in \cite[Lemma 3.3]{Ma5}.

\begin{proposition}\label{onlyL}
The set of all holomorphic functions $\Phi=\log (W \circ F^\mu)'$ in $\widehat{\mathcal B}_p (\mathbb D)$,
given by M\"obius transformations $W$ of $\widehat{\mathbb C}$
and $\mu \in M_p(\mathbb D^*)$, coincides with $L(M_p(\mathbb D^*))$ for every $p \geq 1$.
\end{proposition}

Let $W_a$ be a M\"obius transformation that sends $a \in \widetilde F^\nu(\mathbb D^*)$ to $\infty$. 
Here and in the sequel, $\widetilde F^\nu$ stands for the quasiconformal extension of $F^\nu$ to $\mathbb D^*$.
Since $W_a \circ F^\nu$ is uniquely determined by $\nu \in M_p(\mathbb D^*)$ and $a \in \widetilde F^\nu(\mathbb D^*)$ up to post-composition by affine transformations of $\mathbb C$, 
we can define a map 
$\widetilde L(\nu,a)=\log(W_a \circ F^\nu)' \in L(M_p(\mathbb D^*))$ on 
the fiber space over $M_p(\mathbb D^*)$ given as a domain in the product manifold
$$
\widetilde M_p(\mathbb D^*)=\{(\nu,a) \in M_p(\mathbb D^*) \times \widehat{\mathbb C} 
\mid a \in \widetilde F^{\nu}(\mathbb D^*)\}.
$$
We note that $\widetilde L(\nu,\infty)=L(\nu)$.

The arguments and results in the rest of this section are applied also to different kinds of 
Teichm\"uller spaces (see \cite{Ma4, Ma5}).

\begin{lemma}\label{localbound}
$\widetilde L:\widetilde M_p(\mathbb D^*) \to L(M_p(\mathbb D^*))$ is holomorphic.
\end{lemma}

\begin{proof}
Let $\Phi_0=\widetilde L(\nu,\infty)=\log (F^\nu)'$ and $\Phi=\widetilde L(\nu,a)=\log (W_a \circ F^\nu)'$.
Then a simple computation yields
\begin{align}\label{1st2nd}
\Phi'(z)&=N_{W_a \circ F^\nu}(z)=N_{W_a} \circ F^\nu(z)\cdot (F^\nu)'(z)+N_{F^\nu}(z)
=\frac{-2(F^\nu)'(z)}{{F^\nu}(z)-a}+\Phi_0'(z);\\
\Phi''(z)&=\frac{2(F^\nu)'(z)^2}{(F^\nu(z)-a)^2}
-\frac{2(F^\nu)''(z)}{F^\nu(z)-a}+\Phi_0''(z).
\end{align}
When $a=\infty$, these read as $\Phi'(z)=\Phi_0'(z)$ and $\Phi''(z)=\Phi_0''(z)$. We may assume $a \neq \infty$.
Since $a \in \widetilde F^\nu(\mathbb D^*)$, the denominator $F^\nu(z)-a$ with $z \in \mathbb D$ is 
bounded below by the distance $d(a,\partial F^\nu(\mathbb D))$, which is bounded away from $0$
uniformly in $z$ and locally uniformly in $a$. Hence, it suffices to estimate the norms of $((F^\nu)')^2$ and 
$(F^\nu)''$ for $\Vert \Phi-\Phi_0 \Vert_{\widehat{\mathcal B}_p} \asymp \Vert \Phi-\Phi_0 \Vert_{{\mathcal B}^{\#}_p}
+\Vert \Phi-\Phi_0 \Vert_{{\mathcal B}_\infty}$ (see Remark \ref{normsupply}).

First, we consider the ${\mathcal B}^{\#}_p$-norm:
$$
\Vert \Phi-\Phi_0 \Vert_{{\mathcal B}^{\#}_p}^p\lesssim
\int_{\mathbb D}|(1-|z|^2)^2(F^\nu)'(z)^2|^p\frac{dxdy}{(1-|z|^2)^2}+
\int_{\mathbb D}|(1-|z|^2)^2(F^\nu)''(z)|^p\frac{dxdy}{(1-|z|^2)^2}.
$$
The first term is estimated by
\begin{equation}\label{first}
\int_{\mathbb D}|(1-|z|^2)^2(F^\nu)'(z)^2|^p\frac{dxdy}{(1-|z|^2)^2}=
\int_{F^\nu(\mathbb D)}\delta(\zeta)^{2p-2}d\xi d\eta \lesssim ({\rm diam}(F^\nu(\mathbb D)))^{2p},
\end{equation}
where $\delta$ is the inverse of half the hyperbolic density in $F^{\nu}(\mathbb D)$, that is,
$\delta(F^{\nu}(z))=(1-|z|^2)|(F^{\nu})'(z)|$. We note that $\delta(\zeta)$ is comparable
to the distance $d(\zeta, \partial F^{\nu}(\mathbb D))$ from $\zeta$ to the boundary $\partial F^{\nu}(\mathbb D)$.
For the second term, we apply the Cauchy--Schwarz inequality:
\begin{align}
&\quad\int_{\mathbb D}|(1-|z|^2)^2(F^\nu)''(z)|^p\frac{dxdy}{(1-|z|^2)^2}\\
&=\int_{\mathbb D}\Bigl|(1-|z|^2)\frac{(F^\nu)''(z)}{(F^\nu)'(z)}\Bigr|^p\cdot
|(1-|z|^2)(F^\nu)'(z)|^p\frac{dxdy}{(1-|z|^2)^2}\\
&\leq \Bigl(\int_{\mathbb D}|(1-|z|^2)N_{F^\nu}(z)|^{2p} \frac{dxdy}{(1-|z|^2)^2}\Bigr)^{1/2}
\Bigl(\int_{\mathbb D}|(1-|z|^2)(F^\nu)'(z)|^{2p} \frac{dxdy}{(1-|z|^2)^2}\Bigr)^{1/2}\\
&\lesssim \Bigl(\Vert \Phi_0 \Vert_{\mathcal B_{2p}}\
{\rm diam}(F^\nu(\mathbb D))
\Bigr)^{p}.
\end{align}
Here, we have applied \eqref{first} in the last line.

Next, we consider the ${\mathcal B}_\infty$-norm dominated by ${\mathcal B}_{2p}$-norm 
by Proposition \ref{norm} (i):
\begin{align}
\Vert \Phi-\Phi_0 \Vert_{{\mathcal B}_\infty} &\lesssim \Vert \Phi-\Phi_0 \Vert_{{\mathcal B}_{2p}}\\
&\lesssim \Bigl(\int_{\mathbb D} |(1-|z|^2)(F^\nu)'(z)|^{2p} \frac{dxdy}{(1-|z|^2)^2}\Bigr)^{1/(2p)}
\lesssim {\rm diam}(F^\nu(\mathbb D)),
\end{align}
where \eqref{first} is used again.

By the above computations, we see that $\Vert \widetilde L(\nu,a) \Vert_{\widehat{\mathcal B}_p}$ is bounded 
by a constant
determined in terms of $d(a,\partial F^\nu(\mathbb D))$, $\Vert L(\nu) \Vert_{\widehat{\mathcal B}_p}$, 
$\Vert L(\nu) \Vert_{{\mathcal B}_{2p}}$, and ${\rm diam}(F^\nu(\mathbb D))$.
For a given $(\nu_0,a_0) \in \widetilde M_p(\mathbb D^*)$,
all these quantities vary within a bounded range 
when $\nu \in M_p(\mathbb D^*)$ and $a \in \widetilde F^{\nu_0}(\mathbb D^*)$ move slightly from $(\nu_0,a_0)$.
This shows that $\widetilde L$ is locally bounded.

Under this local boundedness condition, 
if $\widetilde L$ is G\^{a}teaux holomorphic, then it is holomorphic by Proposition \ref{G-holomorphic}.
As shown in \cite[Lemma 6.1]{WM-2}, the G\^{a}teaux holomorphy of $\widetilde L$ follows from the condition that
for each fixed $z \in \mathbb D$,
$\widetilde L(\nu,a)(z)=\log (W_{a} \circ F^{\nu})'(z)$ is 
G\^{a}teaux holomorphic as a complex-valued function. 
By the holomorphic dependence of quasiconformal mappings on the Beltrami coefficients (see \cite[V. Theorem 5]{Ah}),
this can be verified. Thus, $\widetilde L$ is holomorphic on $\widetilde M_p(\mathbb D^*)$.
\end{proof}

Now we state the replacement of Theorem \ref{SL} as follows.

\begin{theorem}\label{bundle}
$J:L(M_p(\mathbb D^*)) \to S(M_p(\mathbb D^*))$ is a holomorphic split submersion for $p \geq 1$. 
\end{theorem}

\begin{proof}
For any $\Phi \in L(M_p(\mathbb D^*))$, let $\Psi_0=J(\Phi) \in S(M_p(\mathbb D^*))$. Then
there exists a
neighborhood $V_{\Psi_0}$ of $\Psi_0$ in $S(M_p(\mathbb D^*))$ 
and a holomorphic map $\sigma:V_{\Psi_0} \to M_p(\mathbb D^*)$ such that
$S \circ \sigma$ is the identity on $V_{\Psi_0}$, as in the case of 
$\mathbb H$ in Proposition \ref{S-section}. Let $\Phi_0=L \circ \sigma(\Psi_0)$, which may be different 
from $\Phi$.
Since $\Phi_0$ can be represented as $\log (F^{\sigma(\Psi_0)})'$, we have
$\Phi=\log (W_a \circ F^{\sigma(\Psi_0)})'$ for some $a \in \widetilde F^{\sigma(\Psi_0)}(\mathbb D^*)$
by Proposition \ref{affine}. 
Namely, $\Phi=\widetilde L(\sigma(\Psi_0),a)$.

Fix this $a$ and define a map
$\widetilde L(\sigma(\cdot),a):V_{\Psi_0} \to L(M_p(\mathbb D^*))$ after shrinking $V_{\Psi_0}$ if necessary.
By Lemma \ref{localbound}, this is a holomorphic map on $V_{\Psi_0}$.
Since $J \circ \widetilde L(\sigma(\Psi),a)=\Psi$ for every $\Psi \in V_{\Psi_0}$, the map
$\widetilde L(\sigma(\cdot),a)$ is a local holomorphic right inverse of $J$ such that $\widetilde L(\sigma(V_{\Psi_0}),a)$ passes through the given 
point $\Phi=\widetilde L(\sigma(\Psi_0),a)$.
This is equivalent to saying that $J$ is a holomorphic split submersion.
\end{proof}

The {\it Bers fiber space} $\widetilde T_p$ over $S(M_p(\mathbb D^*))$ is defined as
$$
\widetilde T_p=\{(\Psi,a) \in S(M_p(\mathbb D^*)) \times \widehat{\mathbb C} \mid \Psi=S(\nu),
\ a \in \widetilde F^{\nu}(\mathbb D^*),\ \nu \in M_p(\mathbb D^*)\}.
$$
Theorem \ref{Bersemb} in the next section identifies $S(M_p(\mathbb D^*))$ with
the Teichm\"uller space $T_p$. We note that the quasidisk $\widetilde F^{\nu}(\mathbb D^*)$ is determined by $\Psi$
independently of the choice of $\nu \in M_p(\mathbb D^*)$ with $S(\nu)=\Psi$.
We define a map $\lambda: \widetilde T_p \to L(M_p(\mathbb D^*))$ by
$\lambda(\Psi,a)=\widetilde L(\nu,a)$ for $S(\nu)=\Psi$. This is well defined independently of
the choice of $\nu$.

Note that the condition $a \in \widetilde F^{\nu}(\mathbb D^*)$ is equivalent to requiring that
$W_a \circ F^{\nu}$ maps $\mathbb D$ onto a bounded domain in $\mathbb C$, and 
that $a=\infty$ if and only if $W_a$ is an affine transformation of $\mathbb C$.
Hence, by Proposition \ref{affine},
$\lambda$ is bijective.
In fact, $\lambda$ is bijective on each fiber. That is, for each $\Psi \in S(M_p(\mathbb D^*))$ with $S(\nu)=\Psi$, 
$\lambda(\Psi,\cdot)$ maps $\widetilde F^{\nu}(\mathbb D^*)$ bijectively onto 
$J^{-1}(\Psi) \subset L(M_p(\mathbb D^*))$.
Here, $J^{-1}(\Psi)$ is a $1$-dimensional complex submanifold of $L(M_p(\mathbb D^*))$ since 
$J$ is a holomorphic split submersion by Theorem \ref{bundle}.

\begin{lemma}
$\lambda:\widetilde T_p \to L(M_p(\mathbb D^*))$ is a biholomorphic homeomorphism.
\end{lemma}

\begin{proof}
Choose any $\Psi_0 \in S(M_p(\mathbb D^*))$, and take $V_{\Psi_0}$ and $\sigma$ as in the proof of Theorem \ref{bundle}.
The restriction of $\lambda$ to the domain
\[
\widetilde V_{\Psi_0}=\{(\Psi,a) \in V_{\Psi_0} \times \widehat{\mathbb C} \mid 
a \in \widetilde F^{\sigma(\Psi)}(\mathbb D^*)\} \subset \widetilde T_p
\]
is explicitly represented as 
$\lambda_\sigma(\Psi,a)=\widetilde L(\sigma(\Psi),a)$. Then
$\lambda_\sigma$ is holomorphic on $\widetilde V_{\Psi_0}$ by 
Lemma \ref{localbound}, and thus $\lambda$ is a holomorphic bijection. 

Moreover, for each fixed $\Psi \in V_{\Psi_0}$, the domain $\widetilde F^{\sigma(\Psi)}(\mathbb D^*)$ of complex dimension $1$
is mapped by $\lambda_\sigma(\Psi,\cdot)$
holomorphically and bijectively onto the complex submanifold $J^{-1}(\Psi) \subset L(M_p(\mathbb D^*))$.
Hence, $\lambda_\sigma(\Psi,\cdot)$ is a biholomorphic homeomorphism.
It follows from this fiberwise property 
that $\lambda^{-1}$ is holomorphic, and thus $\lambda$ is biholomorphic.
\end{proof}

The structure of the space $L(M_p(\mathbb D^*))$ over $S(M_p(\mathbb D^*))$ can be described precisely as follows.

\begin{theorem}\label{bundle2}
$L(M_p(\mathbb D^*))$ is a real-analytic disk bundle over $S(M_p(\mathbb D^*))$ with projection $J$.
\end{theorem}

\begin{proof}
We have seen that $\lambda_\sigma(\Psi,a)=\widetilde L(\sigma(\Psi),a)=\log (W_a \circ F^{\sigma(\Psi)})'$ is a biholomorphic homeomorphism of $\widetilde V_{\Psi_0} \subset \widetilde T_p$.
Using this, we provide the structure of a disk bundle over $S(M_p(\mathbb D^*))$ for $L(M_p(\mathbb D^*))$.
For every $\Psi_0 \in S(M_p(\mathbb D^*))$, define
$$
\ell_\sigma:V_{\Psi_0} \times \mathbb D^* \to J^{-1}(V_{\Psi_0}) \subset L(M_p(\mathbb D^*))
$$
by
$\ell_\sigma(\Psi, \zeta)=\lambda_\sigma(\Psi,\widetilde F^{\sigma(\Psi)}(\zeta))$. 
By Proposition \ref{affine}, $\ell_\sigma$ is a bijection satisfying 
$J \circ \ell_\sigma(\Psi,\zeta)=\Psi$.
Moreover, $\ell_\sigma$ is a real-analytic diffeomorphism since $\lambda_\sigma$ is biholomorphic and
$\widetilde F^{\sigma(\Psi)}$ is real-analytic by Proposition \ref{real-analytic}.
Hence, $\ell_\sigma$ gives a local trivialization for the projection 
$J:L(M_p(\mathbb D^*)) \to S(M_p(\mathbb D^*))$. 
This implies that $L(M_p(\mathbb D^*))$ possesses the structure of
a fiber bundle described in the statement.
\end{proof}

A global section of the bundle projection $J$ can be obtained by using the global real-analytic right inverse $\Sigma$ of
the Schwarzian derivative map $S:M_p(\mathbb D^*) \to S(M_p(\mathbb D^*))$ for $p>1$,
which is given in Proposition \ref{global} for the case of $\mathbb H$.
Replacing the local right inverse $\sigma$ in the proofs of Theorems \ref{bundle} and \ref{bundle2} with this $\Sigma$,
we define a bi-real-analytic map
$$
\ell_\Sigma:S(M_p(\mathbb D^*)) \times \mathbb D^* \to L(M_p(\mathbb D^*))
$$
by
$\ell_\Sigma(\Psi, \zeta)
=\widetilde L(\Sigma(\Psi),\widetilde F^{\Sigma(\Psi)}(\zeta))$. 
Then, in
the real-analytic category, the total space $L(M_p(\mathbb D^*))$ has the product structure,
and the bundle becomes trivial. 

\begin{corollary}\label{product}
Let $p>1$.
Each $\zeta \in \mathbb D^*$ defines a global real-analytic section 
$$
\ell_\Sigma(\cdot, \zeta):S(M_p(\mathbb D^*)) \to L(M_p(\mathbb D^*))
$$
for the holomorphic bundle projection $J:L(M_p(\mathbb D^*)) \to S(M_p(\mathbb D^*))$.
Moreover, the total space $L(M_p(\mathbb D^*))$ is real-analytically equivalent to $S(M_p(\mathbb D^*)) \times \mathbb D^*$
under $\ell_\Sigma$, with $J \circ \ell_\Sigma(\Psi,\zeta)=\Psi$.
\end{corollary}

Finally, we mention the characterization of $M_p(\mathbb D^*)$ in terms of
$\widehat{\mathcal B}_p(\mathbb D)$ and $\mathcal A_p(\mathbb D)$.
The difference in $J$ from the case of $\mathbb H$ does not affect other statements for the disk model substantially, 
and the result parallel to Theorem \ref{conformal} can be stated as follows.

\begin{corollary}[to Theorem \ref{conformal}]\label{disk2}
Let $F:\mathbb D \to \mathbb C$ be a conformal map onto a bounded domain that extends to a quasiconformal
homeomorphism of the extended complex plane $\widehat{\mathbb C}$.
Then, the following conditions are equivalent for every $p \geq 1$:
\begin{enumerate}
\item
$F$ extends quasiconformally to $\mathbb D^*$ so that its complex dilatation is in $M_p(\mathbb D^*)$;
\item
$\log F'$ belongs to $\widehat{\mathcal B}_p(\mathbb D)$;
\item
$S_{F}$ belongs to ${\mathcal A}_p(\mathbb D)$.
\end{enumerate}
\end{corollary}

We note that the equivalence of $(1)$ and $(3)$ follows from that in Theorem \ref{conformal}
by the M\"obius invariance of the Schwarzian derivative.
However, despite the isomorphic relation between $\widehat{\mathcal B}_p(\mathbb H)$ and
$\widehat{\mathcal B}_p(\mathbb D)$ as in Theorem \ref{isomorphism},
the equivalence involving $(2)$ does not follow directly from Theorem \ref{conformal}.
By preparing the disk versions of Theorem \ref{P-holo} and
Lemma \ref{BtoA},
we must repeat the same arguments as in Theorem \ref{conformal}
to obtain Corollary \ref{disk2}.

\section{Structures of integrable Teich\-m\"ul\-ler spaces}\label{5}

The {\it universal Teich\-m\"ul\-ler space} $T$ is the
set of all normalized quasisymmetric homeo\-morphisms $h:\mathbb R \to \mathbb R$
that extend to quasiconformal homeomorphisms $H(\mu):\mathbb H \to \mathbb H$
with complex dilatations $\mu \in M(\mathbb H)$. Via the correspondence from Beltrami coefficients $\mu$
to quasisymmetric homeomorphisms $h$ through $H(\mu)$, we obtain
a map $\pi: M(\mathbb H) \to T$, called the Teich\-m\"ul\-ler projection.
When $\pi(\mu)=\pi(\nu)$, we say that $\mu$ and $\nu$ are {\it Teich\-m\"ul\-ler equivalent}.
An element $h(\mu)=H(\mu)|_{\mathbb R}$ of $T$ 
can be represented by the Teich\-m\"ul\-ler equivalence class $[\mu]$ for $\mu \in M(\mathbb H)$.
We refer to \cite[Chapter 3]{Le} for the basis of the universal Teich\-m\"ul\-ler space.

Let $F^\mu$ denote the normalized conformal homeomorphism of $\mathbb H^-$ that extends quasiconformally to
$\mathbb C$ with complex dilatation $\mu$ on $\mathbb H^+$,
for $\mu \in M(\mathbb H^+)$.
The Schwarzian derivative map 
$S:M(\mathbb H^+) \to {\mathcal A}_\infty(\mathbb H^-)$, where
${\mathcal A}_\infty(\mathbb H^-)$ is the Banach space of all holomorphic functions $\Psi$ on $\mathbb H^-$ with norm
$$
\Vert \Psi \Vert_{\mathcal A_\infty}=
\sup_{z \in \mathbb H^-} |{\rm Im}\,z|^2|\Psi(z)|<\infty,
$$ 
is defined by the
correspondence $\mu \mapsto S_{F^\mu}$,
and the pre-Schwarzian derivative map $L:M(\mathbb H^+) \to {\mathcal B}_\infty(\mathbb H^-)$
by the
correspondence $\mu \mapsto \log (F^\mu)'$. It is well known that both $S$ and $L$ are holomorphic split submersions onto
their images with respect to the norm $\Vert \cdot \Vert_\infty$ of $M(\mathbb H^+)$.

For $\mu$ and $\nu$ in $M(\mathbb H^+)$,
we have $\pi(\mu)=\pi(\nu)$ if and only if $F^\mu|_{\mathbb H^-}=F^\nu|_{\mathbb H^-}$. This induces
well-defined injections $\alpha:T \to {\mathcal A}_\infty(\mathbb H^-)$ with
$\alpha \circ \pi=S$, and $\beta:T \to {\mathcal B}_\infty(\mathbb H^-)$ with
$\beta \circ \pi=L$. We call $\alpha$ the {\it Bers embedding} and $\beta$ the {\it pre-Bers embedding}.
Then, it follows from the property of split submersion that both $\alpha$ and $\beta$ are
homeomorphisms onto their images.

For $p \geq 1$, the {\it $p$-integrable Teich\-m\"ul\-ler space} $T_p$ is defined by $T_p=\pi(M_p(\mathbb H))$.
The topology on $T_p$ is the quotient topology induced by $\pi$ from
that on $M_p(\mathbb H)$ with norm $\Vert \cdot \Vert_p+\Vert \cdot \Vert_\infty$.
Under this stronger topology, the holomorphy of $S:M_p(\mathbb H) \to {\mathcal A}_p(\mathbb H^-)$
and $L:M_p(\mathbb H) \to \widehat{\mathcal B}_p(\mathbb H^-)$ claimed by Proposition \ref{S-holo} and
Theorem \ref{P-holo} is still valid because ${\mathcal A}_p(\mathbb H^-) \subset {\mathcal A}_\infty(\mathbb H^-)$
and $\widehat{\mathcal B}_p(\mathbb H^-) \subset {\mathcal B}_\infty(\mathbb H^-)$ and the inclusions are continuous.
Since $T_p \subset T$, $\alpha$ and $\beta$ are also defined on $T_p$ by the restriction of these maps.
The complex Banach structure on $T_p$ is induced by these embeddings $\alpha$ and $\beta$. Indeed, 
Proposition \ref{S-section}, Theorem \ref{SL}, Corollary \ref{L-section},
and Theorem \ref{conformal} imply the following.

\begin{theorem}\label{Bersemb}
Let $p \geq 1$.
The Bers embedding $\alpha$ is
a homeomorphism onto the open set $\alpha(T_p)=S(M_p(\mathbb H^+))$ in ${\mathcal A}_p(\mathbb H^-)$.
The pre-Bers embedding $\beta$ is
a homeomorphism onto the open set $\beta(T_p)=L(M_p(\mathbb H^+))$ in $\widehat{\mathcal B}_p(\mathbb H^-)$.
These sets are given by
\begin{equation}\label{alphabeta}
\alpha(T_p)=\alpha(T) \cap {\mathcal A}_p(\mathbb H^-),\quad
\beta(T_p)=\beta(T) \cap \widehat{\mathcal B}_p(\mathbb H^-).
\end{equation}
The topological embeddings $\alpha$ and $\beta$ endow $T_p$ with complex Banach structures that are
biholomorphically equivalent.
\end{theorem}

\begin{remark}\label{remark7}
Using $M_p(\mathbb D^*)$ and ${\mathcal A}_p(\mathbb D)$, the Bers embedding $\alpha$ is
defined in the same way and has the same properties as above. However, 
the pre-Bers embedding $\beta$ cannot be defined in this setting, because 
the analogue of Theorem \ref{SL} fails with respect to the injectivity of $J$.
\end{remark}

Next, we consider the metric structure of $T_p$.
In the universal Teich\-m\"ul\-ler space $T$, the Teich\-m\"ul\-ler distance
is defined using the $L_\infty$-norm of Beltrami coefficients: 
the distance from the origin to $[\mu] \in T$ is
the infimum of $\log \bigl((1+\Vert \mu \Vert_\infty)/(1-\Vert \mu \Vert_\infty)\bigr)$ taken over all
Beltrami coefficients $\mu$ in the Teich\-m\"ul\-ler equivalence class $[\mu]$, and this is extended to every point of $T$ by
right translations.
We can provide a similar distance for $T_p$; in particular,
its underlying topological structure is defined in \cite{WM-1} as follows.

\begin{definition}
A sequence $[\mu_n]$ in $T_p$ for $p \geq 1$ converges to $[\nu] \in T_p$ if 
$$
\inf\,\{\Vert \mu_n \ast \nu^{-1}\Vert_p \mid \mu_n \in [\mu_n],\ \nu \in [\nu]\} \to 0 \quad (n \to \infty),
$$
where $ \mu \ast \nu^{-1}$ denotes the complex dilatation of 
the quasiconformal self-homeomorphism $H(\mu) \circ H(\nu)^{-1}$ of $\mathbb H$.
We call this the {\it Teich\-m\"ul\-ler topology}.
\end{definition}

We first show the following. 

\begin{proposition}
For $p \geq 1$,
the Teich\-m\"ul\-ler topology $\mathcal O_p$ on $T_p$ coincides with the quotient topology $\mathcal Q_{p,\infty}$
induced from $M_p(\mathbb H)$ with norm $\Vert \cdot \Vert_p+\Vert \cdot \Vert_\infty$.
\end{proposition}

\begin{proof}
To see that the quotient topology $\mathcal Q_{p,\infty}$ is stronger than $\mathcal O_p$, we show that
$\pi:M_p(\mathbb H^+) \to (T_p,\mathcal O_p)$ is continuous.
For each $[\nu] \in T_p$, there is a representative $\nu \in M_p(\mathbb H^+)$
such that $F^\nu$ is a bi-Lipschitz self-diffeomorphism of $\mathbb H^+$ by \cite[Lemma 3.4]{WM-1},
and for such $\nu$ the convergences $\Vert \mu_n-\nu \Vert_p \to 0$ and 
$\Vert \mu_n \ast \nu^{-1} \Vert_p \to 0$ as $n \to \infty$ are equivalent by \cite[Lemma 3.1]{WM-1}.
Hence, the projection $\pi$ is continuous.

To see that $\mathcal O_p$ is stronger than $\mathcal Q_{p,\infty}$, we show that
the identity map $\iota:(T_p,\mathcal O_p) \to (T_p,\mathcal Q_{p,\infty})$ is continuous.
The fact that $\alpha:(T_p,\mathcal O_p) \to \mathcal A_p(\mathbb H^-)$ is continuous can be verified
by an analogues argument of \cite[I, Lemma 2.9]{TT} with the bi-Lipschitz representative as above.
Since $\alpha(T_p) \subset \mathcal A_p(\mathbb H^-)$ is homeomorphic to $(T_p,\mathcal Q_{p,\infty})$
by Theorem \ref{Bersemb}, the identity map $\iota$ is continuous.
\end{proof}

\begin{remark}
In \cite{WM-1}, a different Teich\-m\"ul\-ler topology $\mathcal O_{p,\infty}$ is used, defined by
replacing $\Vert \mu_n \ast \nu^{-1}\Vert_p$ with
$\Vert \mu_n \ast \nu^{-1}\Vert_p +\Vert \mu_n \ast \nu^{-1}\Vert_\infty$ in the above definition.
Obviously, $\mathcal O_{p,\infty}$ is stronger than $\mathcal O_{p}$.
However, since the continuity of $\pi:M_p(\mathbb H^+) \to (T_p,\mathcal O_{p,\infty})$ 
can be proved in the same way,
the two topologies coincide.
\end{remark}

We now mention the topological group structure of $T_p$.
The Teich\-m\"ul\-ler space $T_p$ (as well as $T$) carries a group structure under the composition of
quasisymmetric homeomorphisms. For $h(\mu)=\pi(\mu)$ and $h(\nu)=\pi(\nu)$ in $T_p$,
the Teich\-m\"ul\-ler equivalence class of
the composition $h(\mu) \circ h(\nu)$ is denoted by $[\mu] \ast [\nu]$, and
the inverse $h(\mu)^{-1}$ by $[\mu]^{-1}$.
For every $[\nu] \in T_p$, the right translation $r_{[\nu]}:T_p \to T_p$
is defined by $[\mu] \mapsto [\mu] \ast [\nu]$.

The following topological-group property
is proved in \cite[Theorem I.3.8]{TT} and \cite[Theorem 6.1]{WM-4}.
The biholomorphic property is shown in \cite[Section 4]{WM-1}.

\begin{proposition}
For $p \geq 1$, $T_p$ is a topological group. 
Moreover,
every right translation $r_{[\nu]}$ is a biholomorphic automorphism of $T_p$.
\end{proposition}

The Weil--Petersson metric on $T_2$ is studied in \cite{Cu} and \cite{TT}.
This metric was generalized to $T_p$ for $p \geq 2$ in \cite{Ma1}.
In fact, the same definition also works for $p \geq 1$.
The {\it $p$-Weil--Petersson metric} on the tangent bundle of $T_p$ is easily defined by embedding $T_p$ 
into $\mathcal A_p(\mathbb H)$ via the Bers embedding $\alpha$ and assuming that
the tangent space ${\mathscr T}_{[\nu]}(T_p)$ of $T_p \cong \alpha(T_p)$ at any point $[\nu] \in T_p$ 
is $\mathcal A_p(\mathbb H)$.
Then, at the origin of $\alpha(T_p)$, the norm of a tangent vector $v$
in ${\mathscr T}_{[0]}(T_p) \cong  \mathcal A_p(\mathbb H)$
is defined to be $\Vert v \Vert_{{\mathcal A}_p}$ (or the norm of the adjoint operator $v^*$
acting on $\mathcal A_q(\mathbb H)$ for $1/p+1/q=1$); see \cite[Section 6.5]{Ma0}. 
For an arbitrary point $[\nu] \in T_p$ with $\alpha([\nu])=\Psi$, consider
the conjugate of the right translation $r_{[\nu]}^{-1}$ by $\alpha$.
Then $\alpha \circ r_{[\nu]}^{-1} \circ \alpha^{-1}$ is a biholomorphic automorphism of $\alpha(T_p)$ sending
$\Psi$ to $0$. The norm of a tangent vector $v$
in ${\mathscr T}_{[\nu]}(T_p) \cong  \mathcal A_p(\mathbb H)$ is defined to be
$\Vert d_\Psi(\alpha \circ r_{[\nu]}^{-1} \circ \alpha^{-1})(v) \Vert_{{\mathcal A}_p}$. This yields 
a Finsler metric on the tangent bundle of $T_p$ in a broad sense.
From the definition, the $p$-Weil--Petersson metric is invariant under the right translations of $T_p$.
The distance induced by this metric is called the {\it $p$-Weil--Petersson distance},
which dominates the Teich\-m\"ul\-ler topology on $T_p$. 

We can also introduce a different invariant Finsler metric using 
the pre-Bers embedding $\beta:T_p \to \widehat{\mathcal B}_p(\mathbb H)$.

\begin{definition}
For any tangent vector $u \in {\mathscr T}_{[\nu]}(T_p) \cong \widehat{\mathcal B}_p(\mathbb H)$ at $[\nu] \in T_p$ 
with $\beta([\nu])=\Phi$ for $p \geq 1$,
the {\it $p$-pre-Weil--Petersson metric} is the Finsler metric 
on the tangent bundle of $T_p$ modeled on
$\widehat{\mathcal B}_p(\mathbb H)$ given by 
$\Vert d_\Phi(\beta \circ r_{[\nu]}^{-1} \circ \beta^{-1})(u) \Vert_{\widehat{\mathcal B}_p}$.
\end{definition}

\begin{theorem}\label{complete}
The integrable Teich\-m\"ul\-ler space $T_p$ for $p \geq 1$ is complete with respect to 
the $p$-pre-Weil--Petersson distance. Moreover, 
the $p$-pre-Weil--Petersson metric is continuous on the tangent bundle of $T_p$ and invariant under the
right translations of $T_p$.
\end{theorem}

\begin{proof}
The proof can be reproduced by mimicking that for the Weil--Petersson metric in
\cite[Theorem 5]{Cu} and \cite[Section 8]{Ma1}.
The only gap for the pre-Bers embedding case is the analogue of the Ahlfors--Weill section
for the Schwarzian derivative map. However, this is successfully filled by
the following claim obtained in \cite[Theorem 5.1]{GH} via the theory of chordal Loewner chains on the half-plane.
Alternatively for the latter statement, since $J:\beta(T_p) \to \alpha(T_p)$ is biholomorphic by Theorem \ref{SL}, 
the results for the Bers embedding transfer directly to the present case.
\end{proof}

\begin{lemma}
If $\Phi \in {\mathcal B}_\infty(\mathbb H^-)$ satisfies $\Vert \Phi \Vert_{{\mathcal B}_\infty}<\tfrac12$,
then $\mu(z)=-2\,{\rm Im}(z)|\Phi'(\bar z)|$ for $z \in \mathbb H^+$
belongs to $M(\mathbb H^+)$ and satisfies $\beta([\mu])=\Phi$.
\end{lemma}

\section{Relationship with Teich\-m\"ul\-ler spaces of diffeomorphisms}\label{6}

In this section, we study the relationship between the integrable Teich\-m\"ul\-ler spaces $T_p$ $(p \geq 1)$
and the Teich\-m\"ul\-ler spaces $T^\gamma$ $(0<\gamma \leq 1)$ of 
orientation-preserving self-diffeomorphisms of $\mathbb R$ and $\mathbb S$, scaled by the regularity of their derivatives.
Since $T^\gamma$ can be characterized by the decay order of the supremum norm of Beltrami coefficients $\mu$ 
(see \cite{Ma2}, \cite{Ma3}, \cite{TW1}, and \cite{TW2}), 
we use this characterization of $T^\gamma$. 
Moreover, because the degeneration of the norm toward $\mathbb R$ and $\mathbb S$ leads to a discrepancy between the Teich\-m\"ul\-ler spaces modeled on $\mathbb H$ and on $\mathbb D$, we restrict attention here to the disk model.

For $0 <\gamma \leq 1$, the space $M^\gamma(\mathbb D^*)$ of $\gamma$-decay
Beltrami coefficients consists of all $\mu \in M(\mathbb D^*)$ such that 
$$
\underset{\ z \in \mathbb D^*}{\mathrm{ess}\sup}\,((|z|^2-1)^{-\gamma} \vee 1)|\mu(z)|<\infty.
$$
Then the Teich\-m\"ul\-ler space $T^\gamma$ of circle diffeomorphisms $h:\mathbb S \to \mathbb S$ whose derivatives $h'$ are
$\gamma$-H\"older continuous turns out to be $\pi(M^\gamma(\mathbb D^*))$. 
For $0<\gamma <1$, this is revealed in \cite[Theorems 1.1, 6.7]{Ma2}, summarizing existing results.
When $\gamma=1$, the corresponding circle diffeomorphisms $h$  
have continuous derivatives $h'$
satisfying the Zygmund condition:
$$
|h'(e^{i(\theta+t)})-2h'(e^{i\theta})+h'(e^{i(\theta-t)})| =O(t) \quad (t \to 0).
$$
The correspondence with $M^1(\mathbb D^*)$ is shown in \cite[Theorem 1.1]{TW1}.

For the image of $M^\gamma(\mathbb D^*)$ under the pre-Schwarzian derivative map $L$, we introduce the space 
${\mathcal B}^\gamma(\mathbb D)$ of $\gamma$-decay Bloch functions $\Phi \in \mathcal B_\infty(\mathbb D)$
satisfying 
$$
\sup_{z \in \mathbb D}\ (1-|z|^2)^{2-\gamma}|\Phi''(z)|<\infty.
$$
When $0<\gamma<1$, this is equivalent to
$\sup_{z \in \mathbb D}(1-|z|^2)^{1-\gamma}|\Phi'(z)|<\infty$. 
As before, $L$ is defined by $L(\mu)=\log (F^\mu)'$ on $\mathbb D$, 
where $F^\mu$ is the normalized conformal homeomorphism of $\mathbb D$ onto a bounded domain that
extends quasiconformally to
$\mathbb C$ with complex dilatation $\mu$ on
$\mathbb D^*$. 

It is proved in \cite[Theorem 4.6]{Ma2} and \cite[Theorem 1.1]{TW1} that 
$L(M^\gamma(\mathbb D^*)) \subset {\mathcal B}^\gamma(\mathbb D)$.
Moreover,
\begin{equation}\label{L-property}
L(M^\gamma(\mathbb D^*))=L(M(\mathbb D^*)) \cap {\mathcal B}^\gamma(\mathbb D),
\end{equation}
and, in particular, there exists a neighborhood of the origin in ${\mathcal B}^\gamma(\mathbb D)$ contained in
$L(M^\gamma(\mathbb D^*))$.
This is the unique component in ${\mathcal B}^\gamma(\mathbb D)$ arising from pre-Schwarzian derivative maps 
with different normalizations of $F^\mu$; that is, the analogue of Proposition \ref{onlyL} holds.
See \cite[Theorem 1.3]{TW1} and \cite[Theorem 1.1]{TW2}.

A basic relation between $T^\gamma$ and $T_p$ is as follows.

\begin{proposition}
If $\gamma p>1$, then $T^\gamma \subset T_p$. In particular, $T^1 \subset T_p$ for all $p>1$.
\end{proposition}

\begin{proof}
This follows from the inclusion $M^\gamma(\mathbb D^*) \subset M_p(\mathbb D^*)$ when $\gamma p>1$,
which is verified by a direct estimate.
\end{proof}

Thus, for $\gamma p>1$ we have the inclusion diagram
\begin{align*}
T^1 \subset &\ T^\gamma \subset \cdots \subset \lim_{\gamma \searrow 0}T^\gamma \quad \text{(decay order)}\\
&\ \cap \qquad \qquad \ \ \cap\\
T_1 \subset &\ T_p \subset \cdots \subset \lim_{p \nearrow \infty}T_p \quad \text{(integrability)}.
\end{align*}

We focus on the relation between $T^1$ and $T_1$. It is shown in \cite{AB} that
every quasisymmetric homeomorphism in the $1$-integrable Teich\-m\"ul\-ler space $T_1$ is
a $C^1$-diffeomorphism of $\mathbb S$ onto itself with nonvanishing derivative.
One might expect $T^1 \subset T_1$, but this is not the case.

\begin{theorem}\label{inclusion}
There is no inclusion relation between $T^1$ and $T_1$. 
\end{theorem}

\begin{proof}
It is shown in \cite[p.17]{HW} that
$\mathcal B^1(\mathbb D)$ and $\mathcal B_1^\#(\mathbb D)$ are incomparable. 
More explicitly, $\Phi^1(z)=a\sum_{n=0}^\infty 2^{-n}z^{2^n}$ belongs to 
$\mathcal B^1(\mathbb D) \setminus \mathcal B_1^\#(\mathbb D)$, while
$\Phi_1(z)=a(1-z)(\log 1/(1-z))^2$ belongs to $\mathcal B_1^\#(\mathbb D) \setminus \mathcal B^1(\mathbb D)$ for
any constant $a \in \mathbb C$.
In Remark \ref{normHD}, we observed that
$\mathcal B_1^\#(\mathbb D)=\widehat{\mathcal B}_1(\mathbb D)$. 

For the pre-Schwarzian derivative map $L$
defined on $M_1(\mathbb D^*)$, we have
\begin{equation}
L(M_1(\mathbb D^*))=L(M(\mathbb D^*)) \cap \widehat{\mathcal B}_1(\mathbb D),
\end{equation}
which follows from Corollary \ref{disk2}.
Combining this with \eqref{L-property}, we see that by choosing $a>0$ sufficiently small, both $\Phi^1$ and $\Phi_1$ lie in $L(M(\mathbb D^*))$ and satisfy
$$
\Phi^1 \in L(M^1(\mathbb D^*)) \setminus L(M_1(\mathbb D^*)), \qquad 
\Phi_1 \in L(M_1(\mathbb D^*)) \setminus L(M^1(\mathbb D^*)).
$$

Applying $J:L(M(\mathbb D^*)) \to S(M(\mathbb D^*))$, which is not injective, we claim that
\begin{equation}\label{belonging}
J(\Phi^1) \in S(M^1(\mathbb D^*)) \setminus S(M_1(\mathbb D^*)), \qquad
J(\Phi_1) \in S(M_1(\mathbb D^*)) \setminus S(M^1(\mathbb D^*)).
\end{equation}
These two conditions yield the theorem, because $S(M^1(\mathbb D^*))$ is identified with $T^1$ via the Bers embedding
$\alpha:T^1 \to S(M^1(\mathbb D^*))$ by \cite[Theorem 3]{Ma3}, while $S(M_1(\mathbb D^*))$ is identified with $T_1$ by Theorem \ref{Bersemb} and
Remark \ref{remark7}.

It remains to prove \eqref{belonging}.
Set $\Phi^1=\log (F^\mu)'$ with $\mu \in M^1(\mathbb D^*)$ and
$J(\Phi^1)=S(\mu)$. 
Suppose, toward a contradiction, that $J(\Phi^1) \in S(M_1(\mathbb D^*))$. Then there exists $\nu \in M_1(\mathbb D^*)$ such that
$S(\nu)=S(\mu)$. 
Proposition \ref{affine} (i)
yields a M\"obius transformation $W$ of $\widehat{\mathbb C}$ such that
$\Phi^1=\log (W \circ F^{\nu})'$ and $W \circ F^{\nu}(\mathbb D)$ is bounded, and then (ii) implies
$\Phi^1 \in L(M_1(\mathbb D^*))$. However, this contradicts $\Phi^1 \notin L(M_1(\mathbb D^*))$.
Thus $J(\Phi^1) \notin S(M_1(\mathbb D^*))$, proving the first inclusion in \eqref{belonging}.
The second follows by the same argument.
\end{proof}

\end{document}